\documentclass[11pt,reqno]{amsart}

\usepackage{amsmath}
\usepackage{amssymb}
\usepackage{amsfonts}

\usepackage[left=30mm,right=30mm,top=30mm,bottom=30mm,marginparwidth=70pt]{geometry}
\usepackage{mathrsfs}
\usepackage{hyperref}
\usepackage{cite}
\usepackage[foot]{amsaddr}
\usepackage[sc]{mathpazo}
\usepackage{graphics}
\usepackage{graphicx}
\usepackage{color}
\usepackage{empheq}
\usepackage{comment}
\usepackage{marginnote}

\newcommand{\Z}{\mathbb{Z}}
\newcommand{\N}{\mathbb{N}}
\newcommand{\R}{\mathbb{R}}
\newcommand{\C}{\mathbb{C}}
\newcommand{\I}{\mathbb{I}}
\renewcommand{\L}{\mathbb{L}}

\newcommand{\B}{\mathscr{B}}

\newcommand{\D}{\mathbf{D}}

\newcommand{\h}{\mathfrak{h}}
\renewcommand{\H}{\mathscr{H}}

\newcommand{\J}{\mathscr{J}}

\newcommand{\Ga}{\mathbf{\Gamma}}

\renewcommand{\rho}{\varrho}

\DeclareMathOperator{\dom}{dom}

\DeclareMathOperator{\dist}{dist}
\DeclareMathOperator{\diam}{diam}
\DeclareMathOperator{\area}{area}
\DeclareMathOperator{\vol}{vol}

\newcommand{\la}{\langle}
\newcommand{\ra}{\rangle}

\newcommand{\eps}{\varepsilon}
\newcommand{\e}{_{\varepsilon}}
\newcommand{\ke}{_{k,\varepsilon}}
\newcommand{\ie}{_{i,\varepsilon}}
\newcommand{\je}{_{j,\varepsilon}}
\newcommand{\ije}{_{i,j,\varepsilon}}

\newcommand{\al}{\alpha}

\renewcommand{\d}{\,\mathrm{d}}
\newcommand{\x}{\mathbf{x}}
\newcommand{\ds}{\displaystyle}

\newcommand{\Id}{\mathrm{Id}}

\newcommand{\cupl}{\bigcup\limits}
\newcommand{\suml}{\sum\limits}

\newcommand{\wt}{\widetilde}
\newcommand{\wh}{\widehat}

\newcommand{\ceq}{\coloneqq}
\newcommand{\restr}{\!\restriction}

\renewcommand{\a}{\mathfrak{a}}
\newcommand{\A}{\mathscr{A}}

\theoremstyle{plain}
\newtheorem{theorem}{Theorem}[section]
\newtheorem*{theorem*}{Theorem}
\newtheorem{lemma}[theorem]{Lemma}
\newtheorem*{lemma*}{Lemma}

\theoremstyle{remark}
\newtheorem{remark}[theorem]{Remark}
\newtheorem*{remark*}{Remark}
\newtheorem{example}[theorem]{Example}
\newtheorem*{example*}{Example}
\theoremstyle{definition}

\numberwithin{equation}{section}
\numberwithin{figure}{section}

\title
[A geometric approximation of non-local interface and boundary conditions]
{A geometric approximation of non-local interface and boundary conditions}

\author[Pavel Exner]{Pavel Exner\,$^{1,2}$}
\address{$^1$ Doppler Institute for Mathematical Physics and Applied Mathematics, Czech Technical University,  B\v rehov\'a 7, 11519 Prague, Czechia}
\address{$^2$ Department of Theoretical Physics, Nuclear Physics Institute, Czech Academy of Sciences, Hlavn\'{\i} 130, 25068 \v{R}e\v{z} near Prague, Czechia}
\email{exner@ujf.cas.cz}

\author[Andrii Khrabustovskyi]{Andrii Khrabustovskyi\,$^3$}
\address{$^3$ Department of Physics, Faculty of Science, University of Hradec Kr\'{a}lov\'{e}, Rokitansk\'eho 62, 50003 Hradec Kr\'alov\'e, Czech Republic}

\email{andrii.khrabustovskyi@uhk.cz}

\begin{document}
	\allowdisplaybreaks	
	
	\begin{abstract}
		We analyze an approximation of a Laplacian subject to non-local interface conditions of a $\delta'$-type by Neumann Laplacians on a family of Riemannian manifolds with a sieve-like structure. We establish a (kind of) resolvent convergence for such operators, which in turn implies the convergence of spectra and eigenspaces, and demonstrate convergence of the corresponding semigroups. Moreover, we provide an explicit example of a manifold allowing to realize any prescribed integral kernel appearing in that interface conditions. Finally, we extend the discussion to similar approximations for the Laplacian with non-local Robin-type boundary conditions.
	\end{abstract}
	
	
	
	\maketitle	
	
	\section{Introduction}\label{sec:1}
	
	Schr\"odinger operators with singular potentials (interactions) supported on hypersurfaces have {attracted} significant attention over the past decade, driven by their relevance in quantum mechanics -- particularly in the modeling of thin-layer and surface-confined quantum systems -- as well as their role as natural mathematical generalizations of one-dimensional point interactions. The presence of such a singular interaction implies that the wave function must satisfy appropriate interface conditions on the underlying surface (which we denote by $\Gamma$). Among the various possible conditions, two have received particular attention: the so-called $\delta$-interactions, characterized by the continuity of the wave function across $\Gamma$ with a jump in its normal derivative, and the $\delta'$-interactions, in which the roles are reversed -- the normal derivative remains continuous, while the wave function itself exhibits a discontinuity. More precisely, the interface conditions corresponding to a $\delta'$-interaction read as follows,
	\begin{align}\label{delta'}
		\left({\partial u\over\partial\nu}\right)^+
		=
		\left({\partial u\over\partial\nu}\right)^-
		=
		\gamma(u^+-u^-)
	\end{align}
	where   $u^\pm$ denote the traces of the wave function on either side of  $\Gamma$,	$\left({\partial u \over\partial   \nu}\right)^{\pm}$ are the corresponding traces of the normal derivative taken with respect to the unit normal vector field directed from the ‘minus’ side of $\Gamma$ towards the ‘plus’ side, $\gamma$ represents the interaction strength, 
	interaction strength, repulsive for $\gamma>0$ and attractive for $\gamma<0$, given as
	a function on $\Gamma$. For a detailed analysis of both models, we refer the reader to \cite{BLL13a,BEL14}.
	
	While these models are both useful and mathematically tractable, it is important to remember that the singular interaction represents an idealized approximation of a more realistic physical description. Consequently, a central problem in this area is to understand how such interactions can be approximated by regular models -- for example, by Schrödinger operators with smooth potentials supported in shrinking neighborhoods of the surface $\Gamma$.
	Such approximations are well known for $\delta$-interactions (see \cite{BEHL17,EI01,EK03,AGS87,Si92}), whereas for $\delta'$-interactions, a result on approximation by regular Schrödinger operators is known only in the one-dimensional case in which $\Gamma$ is a point and the normal derivatives in \eqref{delta'} are replaced by the usual derivative -- see {\cite{AN00,ENZ01} for the result with full mathematical rigor   following a seminal physicists' idea \cite{CS98}}. However, {also another kind of regular approximations of $\delta'$-interactions is known} -- as a homogenization limit of the Neumann Laplacian on a suitable perforated domain $\Omega\e$. This domain has the form $\Omega\setminus\Gamma\e$, where $\Omega$ is a domain on which our operator of interest is given (it is intersected by $\Gamma$), while $\Gamma\e$ is either a subset of $\Gamma$ obtained by drilling a lot of small holes in it (\emph{thin sieve}) or an $\eps$-neighbourhood of $\Gamma$ punctured by many narrow passages (\emph{thick sieve\footnote{The term `thick sieve' may seem misleading, as the sieve thickness tends to zero as $\varepsilon \to 0$. It was introduced in \cite{DelV87}, presumably to distinguish this setting from the earlier studies of sieves with   zero thickness.}}), see Figure~\ref{fig0}. As $\varepsilon \to 0$, the number of holes (respectively, passages) tends to infinity, while their diameters tend to zero. It is known that, if the holes (respectively, passages) are appropriately scaled, the limiting behavior yields the desired $\delta'$-interaction -- see \cite{AP87,Da85,Kh23,MS66,Mu85,Pi87} (respectively, \cite{DelV87,Kh24}). Needless to say, such geometric approximations are only possible when the approximated operator is non-negative (since {any} Neumann Laplacian is {by definition} a non-negative operator); this is the case when the interaction strength {in \eqref{delta'}} is a non-negative function.
	\begin{figure}
		\centering
		\includegraphics[width=0.3\linewidth]{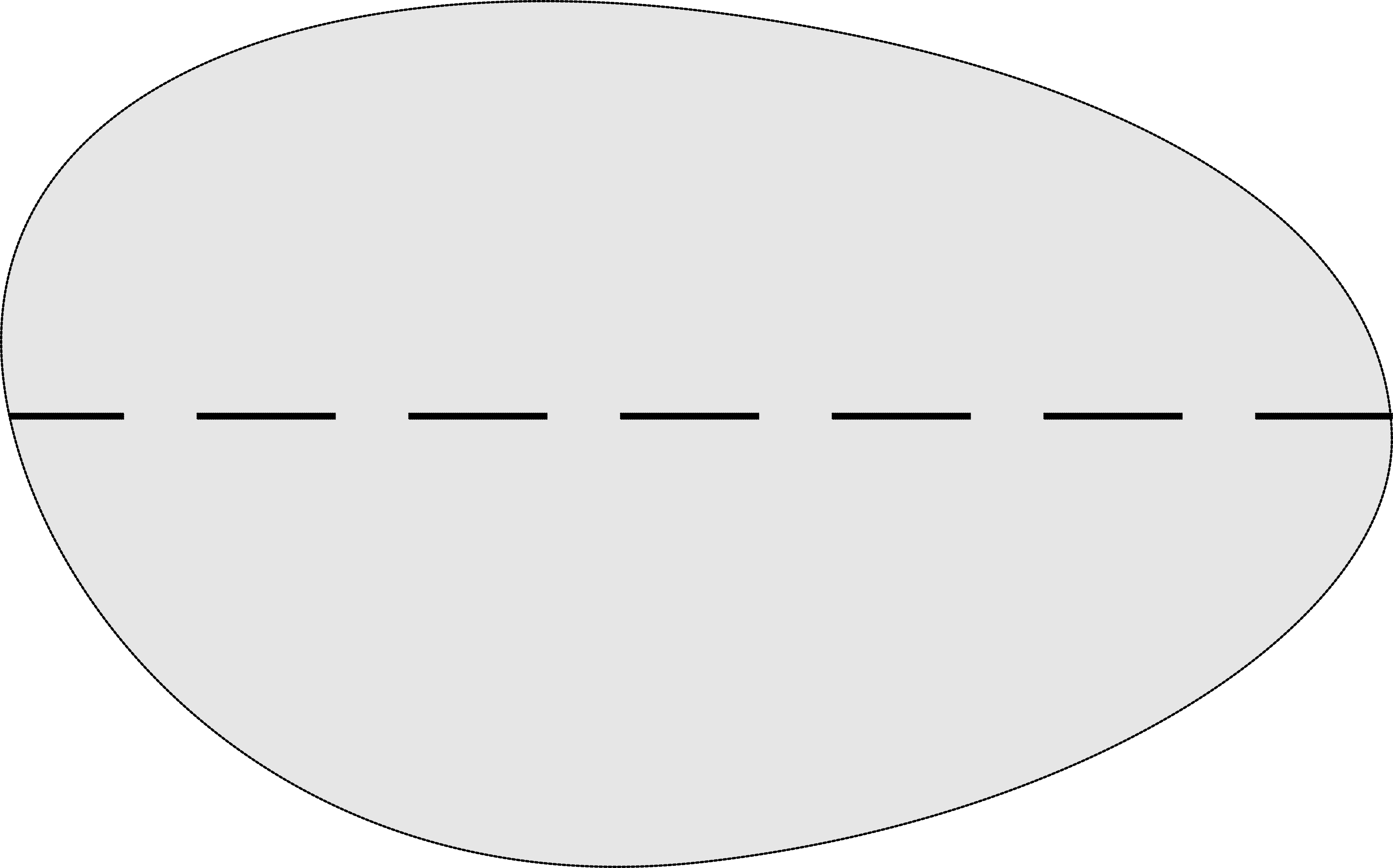}\qquad\qquad
		\includegraphics[width=0.3\linewidth]{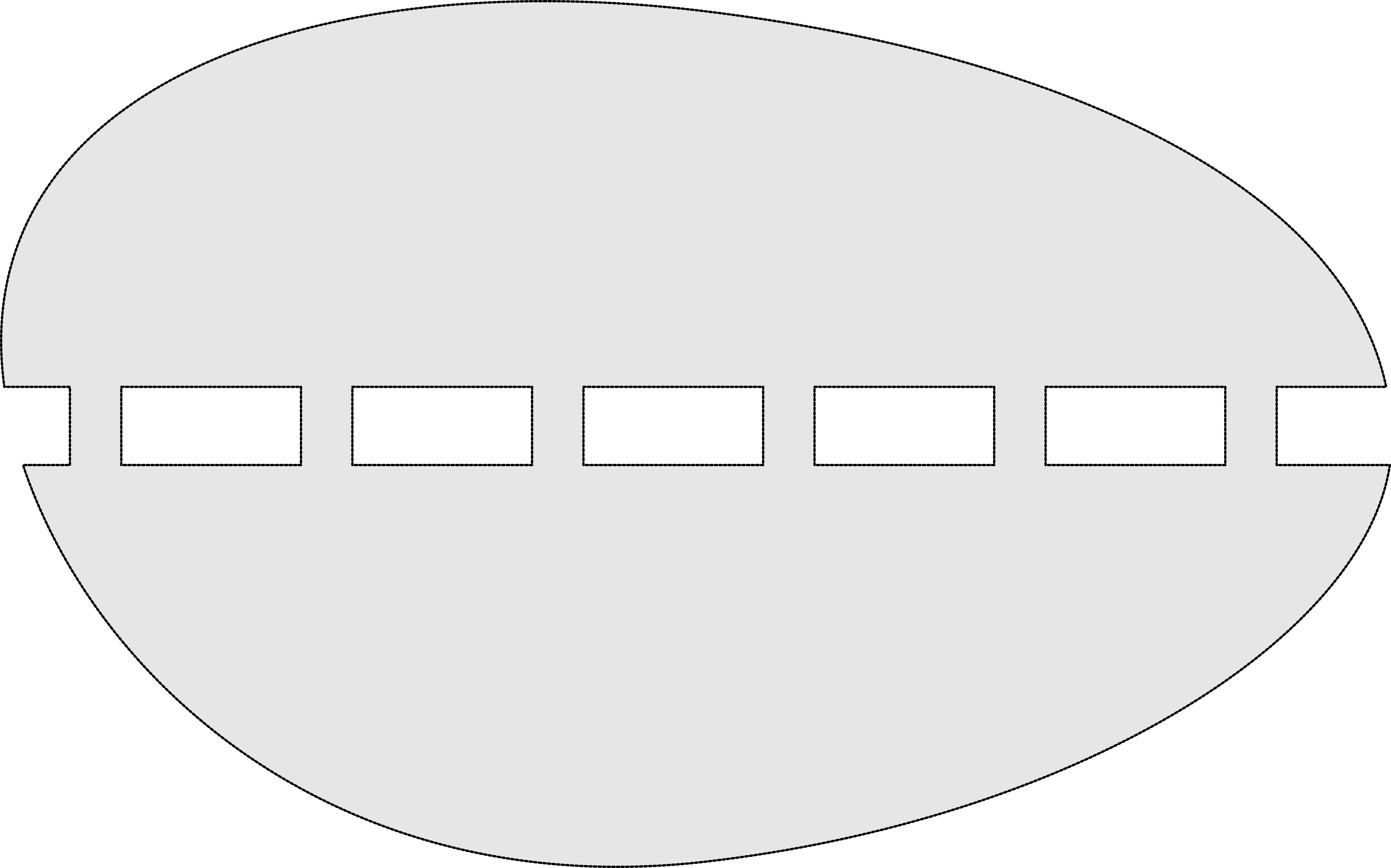}
		\caption{The domain $\Omega$ perforated by a thin sieve (left image) and a thick sieve (right image). Under suitable assumptions on {the} scalings of the sieves, the Neumann Laplacian on this domain approximates the Laplacian on $\Omega$ subject to the `standard' $\delta'$-type conditions \eqref{delta'}.}
		\label{fig0}
	\end{figure}
	
	\smallskip
	
	In the current paper we address the Laplacian subject to the non-local counterpart of the $\delta'$ conditions \eqref{delta'}, namely,
	\begin{align}\label{delta':nonlocal}
		\ds\left(\frac{\partial u}{\partial \nu}\right)^\pm=\pm\int_\Gamma K(x,y)(u^\pm(x)-u^\mp(y)){\d s_y}
	\end{align}
	with a non-negative symmetric continuous kernel $K(x,y)$; we denote this operator by $\H$. Note that if we (formally) put $K(x,y)=\gamma(x)\delta_x(y)$ in \eqref{delta':nonlocal}, where $\delta_x(y)$ is a Dirac $\delta$-function on $\Gamma$ supported
	at $x\in\Gamma$, we arrive at the `standard' $\delta'$-interaction \eqref{delta'}. The goal of this paper is to construct approximations of this operator by Neumann Laplacians.
	
	The nonlocal nature of the interface conditions \eqref{delta':nonlocal} suggests {to consider} a thick sieve, but with passages connecting {general points, not just mirror images on the two sieve borders. This concept has a natural physical appeal. If two regions are connected through a thick `porous' interface, the passages are rarely straight unless the sieve material is a product of a sophisticated engineering. Implementing this ideas mathematically, however,} we immediately face difficulties. First {of all}, in the two-dimensional case, intersections between the passages {become then} unavoidable, whereas in dimensions $n \geq 3$, although {they} can be avoided, {we obtain in general a highly complex} sieve geometry. {Secondly}, if we allow the passages to connect distant points, their lengths do not vanish in the limit -- even though they lie within a shrinking neighborhood of $\Gamma$. However, to get $\delta'$-type interactions the passages must be short
	(cf.~\cite{DelV87,Kh24}).
	
	To avoid these {troubling features}, we consider the Neumann Laplacian not on a domain in $\mathbb{R}^n$, but on a Riemannian manifold. This manifold is constructed as a union of domains in $\mathbb{R}^n$  with parts of their boundaries identified in such a way that the resulting space forms a topological manifold with boundary. Essentially, this manifold resembles a sieve, yet it cannot be regarded as a domain in $\mathbb{R}^n$.
	
	\smallskip
	
	Along with  the Laplacian with interface conditions \eqref{delta':nonlocal}, we also {consider} the Laplacian subject to non-local Robin-type boundary conditions of the form
	\begin{align}\label{Robin:nonlocal}
		\ds \frac{\partial u}{\partial \nu} =
		-\int_\Gamma K(x,y)(u(x)-u(y)){\d s_y}.
	\end{align}
	Here $\Gamma$ is a  subset of $\partial\Omega$,
	$\nu$ is the outward unit normal to $\partial\Omega$. The analysis of Laplacians with abstract non-local boundary conditions (covering those mentioned above) has been carried out in \cite{GM09,BLL13b}. Our objective remains the same: to approximate such operators by Neumann Laplacians. In this case, the approximating manifold consists of a domain $\Omega$ connected by numerous short thin `handles' -- see Theorem~\ref{th:Robin} for {precise statement of the} result.
	
	Elliptic operators with non-local boundary conditions frequently serve as infinitesimal generators of Feller semigroups associated with diffusion processes  -- see the pioneering works \cite{Fe52,Fe54,We59} and some key follow-ups \cite{SU65,GS01,Ta92}. In particular, diffusion equations with non-local Robin   conditions
	resembling those in \eqref{Robin:nonlocal} were addressed in \cite{AKK18}. In this context, our homogenization result (cf. \eqref{diffusion} in Theorem~\ref{th:Robin}) provides an interpretation of the boundary conditions \eqref{Robin:nonlocal}: a diffusing particle reaches the boundary {from which it returns to the bulk}, but due to the presence of attached passages connecting distant points, {its motion continues starting} from a different location.
	
	The homogenization problem considered in this paper is closely related to the homogenization of Laplacians on Riemannian manifolds consisting of a base manifold with numerous attached `handles' or (in the case of a multicomponent base manifold) `wormholes' connecting the base manifold components. Such problems have been studied in \cite{Kh09,Kh13,BK96,AP25}. The key difference between the manifolds analyzed in those works and the one we consider here lies in the placement of the handles and wormholes: in previous studies, they are attached throughout the interior of the manifold, whereas in our model, they are attached along the boundary. In both settings, the homogenized operator may be non-local (cf.~\cite{Kh09,BK96}). However, in contrast to our case -- where non-locality arises in the boundary conditions -- the non-local terms in the cited works appear in the operator's action itself.
	\smallskip
	
	The paper is organized as follows. In Section~\ref{sec:2}, we rigorously define the Laplacian with the interface conditions \eqref{delta':nonlocal}, construct the approximating manifold, and formulate the main result (Theorem~\ref{th1}) which establishes the resolvent convergence of the associated operators. As a consequence, we also state the convergence of eigenvalues and eigenfunctions (Theorem~\ref{th2}) and the convergence of  the semigroups (Theorem~\ref{th3}). Section~\ref{sec:3} presents an example of a manifold that satisfies the assumptions we claim and yields an arbitrary prescribed integral kernel $K(x,y)$ in the homogenization limit. The proofs of Theorems~\ref{th1}, \ref{th2}, \ref{th3} are provided in Section~\ref{sec:4}. Finally, in Section~\ref{sec:5}, we discuss the approximation of the Laplacian with Robin-type non-local boundary conditions.
	
	\section{Problem setting and main results}\label{sec:2}
	\thispagestyle{empty}
	
	Let $\Omega\subset\R^n$ be a bounded Lipschitz domain containing the origin. We denote by $\Gamma$ the intersection of $\Omega$ with a hyperplane $\{x=(x^1,\dots,x^n)\in\R^n:\ x^n=0\}$. By $\Omega^+$ (respectively, $\Omega^-$)
	we denote the subsets of $\Omega$ lying {`above' and `below'} $\Gamma$ {in sense that $\pm x^n>0$, respectively}, see Figure~\ref{fig1} (left).
	We assume that both $\Omega^+$ and $\Omega^-$ satisfy the cone property \cite[Chapt.~4]{AF03}. This assumption is merely technical, serving to ensure the compactness of the embedding $H^1(\Omega^\pm) \hookrightarrow L^2(\Omega^\pm)$ (the Rellich-Kondrashov theorem) as well as the compactness of the trace operators $H^1(\Omega^\pm) \to L^2(\Gamma)$ --- see, e.g., \cite[Theorem~{6.3}]{AF03}, which covers both statements.
	We also introduce the domain $\Ga \subset\R^{n-1}$ by
	\begin{gather}\label{Ga}
		\Ga \ceq\{\x=(x^1,\dots,x^{n-1})\in\R^{n-1}: (\x,0)\in\Gamma\}.
	\end{gather}
	In $L^2(\Omega)$ we consider the sesquilinear form $\h$ given by
	\begin{align*}
		\h[u,v]&=\int_{\Omega^+}\nabla u\cdot\overline{\nabla v}\d x+
		\int_{\Omega^-}\nabla u\cdot\overline{\nabla v}\d x\\
		&+
		\int_\Gamma\int_\Gamma K(x,y)(u^+(x)-u^-(y))\overline{(v^+(x)-v^-(y))}
		\d s_x \d s_y
	\end{align*}
	on the domain $\dom(\h)=H^1(\Omega\setminus\Gamma)$. Here  $K:\overline{\Gamma}\times\overline{\Gamma}\to\R$ is a continuous function satisfying
	\begin{gather}\label{K:prop}
		K(x,y)=K(y,x),\quad K(x,y)> 0,
	\end{gather}
	$\d s_*$ stands for the surface element on $\Gamma$ (with the subscript indicating the variable of integration), $u^+$ and $u^-$ are the traces of $u$ being taken from $\Omega^+$ and $\Omega^-$, respectively. Standard {argument using} trace theorem yields $u^\pm\in H^{1/2}(\Gamma)\subset L^2(\Gamma)$. Evidently, the form $\h$ is well-defined, it is symmetric, densely defined, and positive. 
	Furthermore, the continuity of the trace operators $H^1(\Omega^\pm)\to L^2(\Gamma)$ implies that the energy norm $$\|f\|_{\h} \ceq \left(\h[f, f] + \|f\|_{L^2(\Omega)}^2\right)^{1/2}$$ 
	is equivalent to the standard Sobolev norm on $H^1(\Omega\setminus\Gamma)$, 
	whence the form $\h$ is closed.
	Then by \cite[Theorem 6.2.1]{K66} there exists {a} unique self-adjoint and positive operator $\H$ associated with $\h$, i.e. {we have} $\dom(\H)\subset\dom(\h)$ and
	\begin{gather*}
		\forall u\in
		\dom(\H),\ \forall  v\in \dom(\h):\ (\H u,v)_{L^2(\Omega)}= \h[u,v].
	\end{gather*}
	The spectrum of $\H$ is purely discrete because the form domain $\dom(\h)$, endowed with the energy norm $\|f\|_{\h},$ is compactly embedded in $L^2(\Omega)$. This result follows from the  equivalence of $\|\cdot\|_{\h}$ to the standard norm on $H^1(\Omega\setminus\Gamma)$ and the Rellich-Kondrashov  theorem.
	
	Our main goal is to approximate the resolvent $(\H+\Id)^{-1}$ of the operator $\H$ by the resolvent $(\H\e+\Id)^{-1}$ of the Neumann Laplacian $\H\e$ on an appropriately constructed `sieve-type' Riemannian manifold $M\e$; hereinafter, by $\Id$ we denote the identity operator. Note that $ u=(\H+\Id)^{-1}f$ with $f\in L^2(\Omega)$ is the weak solution to the following boundary value problem:
	\begin{gather*}
		\begin{cases}
			-\Delta u^\pm + u^\pm = f^\pm&\text{in }\Omega^\pm,\\
			\ds\frac{\partial u^\pm}{\partial x^n}=\pm\int_\Gamma K(x,y)(u^\pm(x)-u^\mp(y)){\d s_y}&\text{on }\Gamma,\\[2mm]
			\ds\frac{\partial u}{\partial \nu}=0&\text{on }\partial\Omega.
		\end{cases}
	\end{gather*}
	Here $u^+,\,f^+$ and $f^-,\,f^-$ denote
	the restrictions of $u,\,f$ on $\Omega^+$ and $\Omega^-$, respectively;
	the same notations $u^\pm,\,f^\pm$ are used for the traces of $u^\pm,\,f^\pm$ on $\Gamma$.
	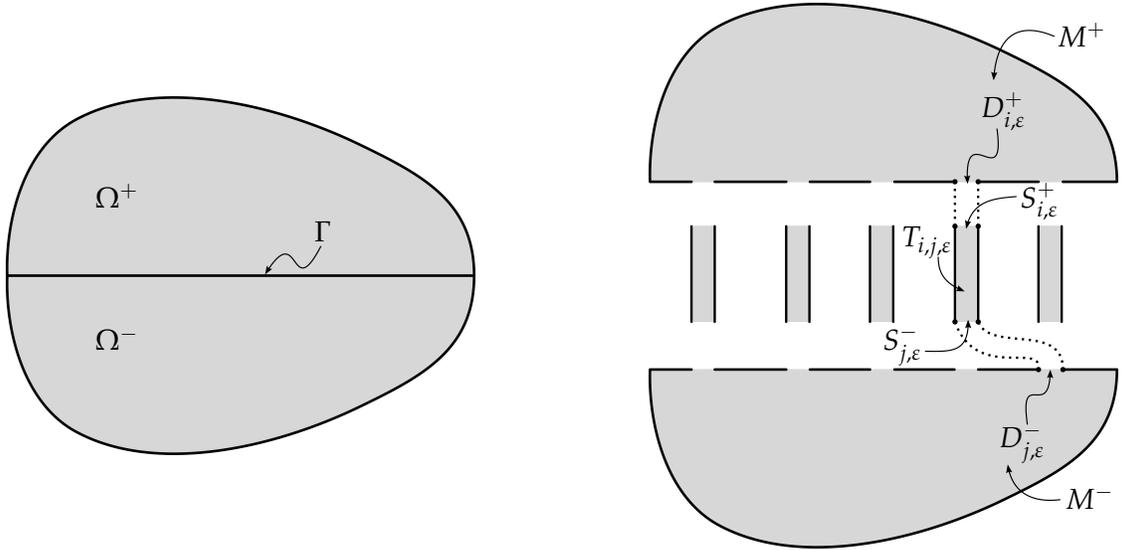
\begin{figure}[t]
		\centering
		\raisebox{-0.5\height}{\includegraphics[width=0.4\textwidth]{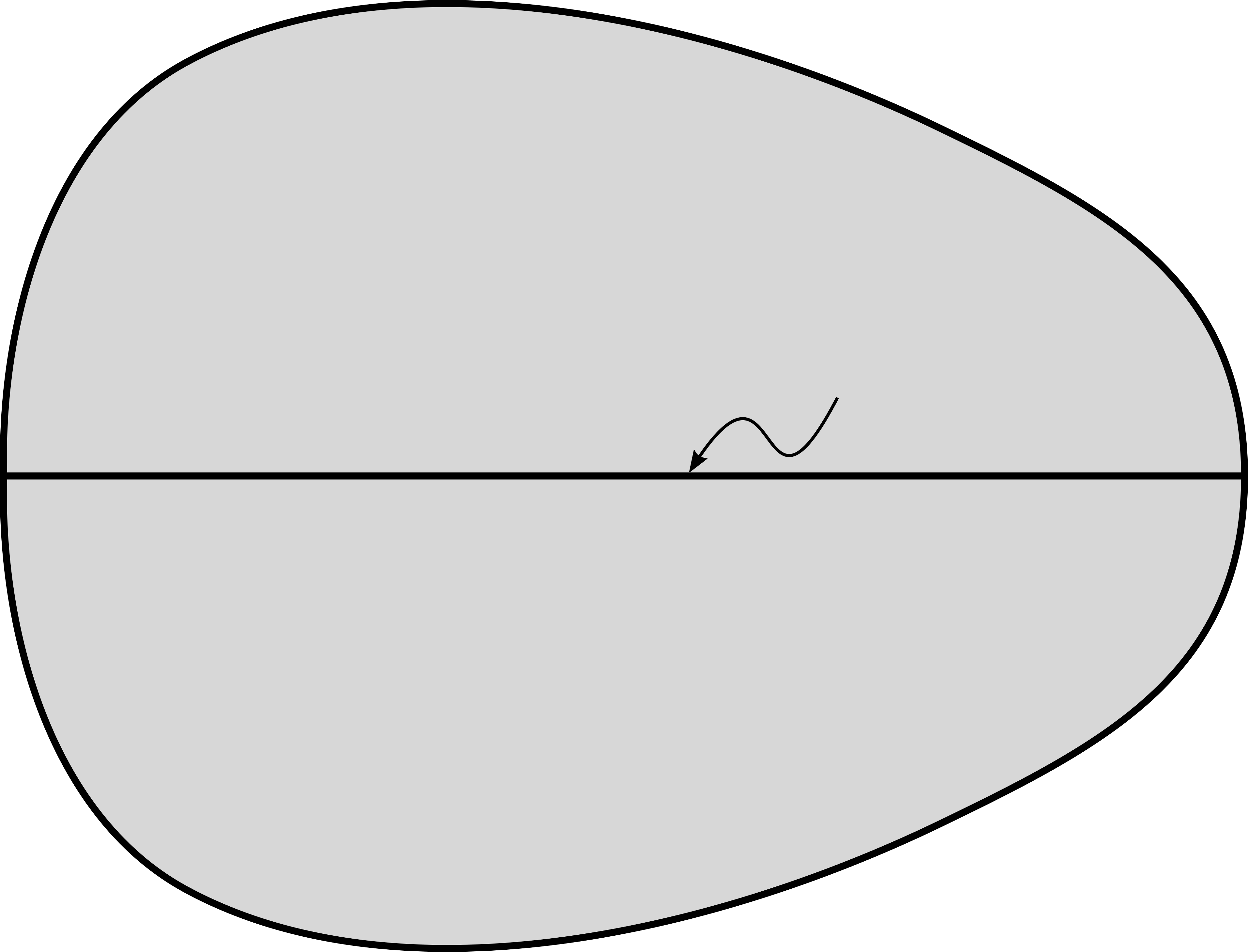}}\qquad\qquad\qquad
		\raisebox{-0.5\height}{\includegraphics[width=0.4\textwidth]{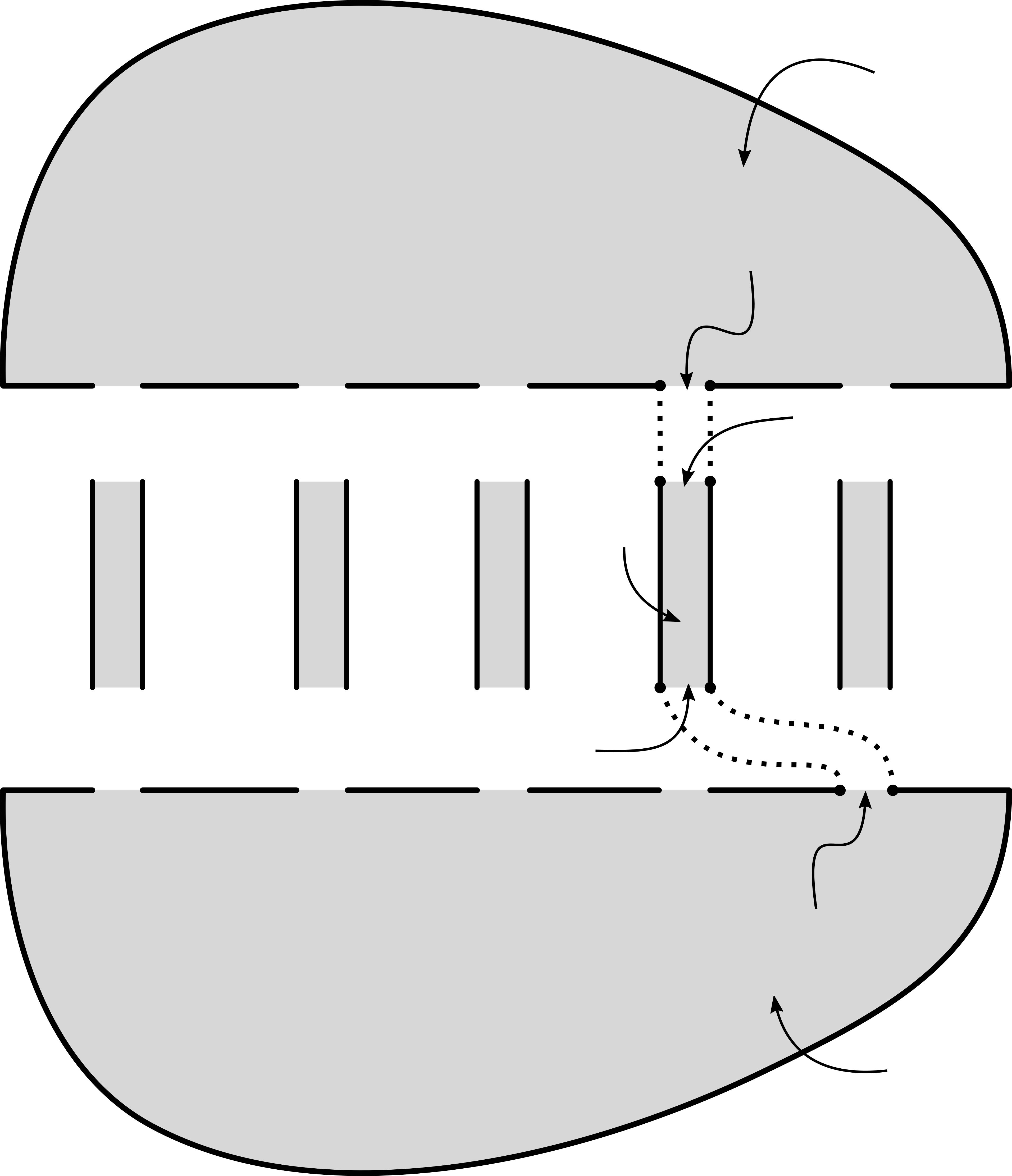}}
		\begin{picture}(0,0)
			\put(-26,87){$M^+$}
			\put(-23,-88){$M^-$}
			\put(-307,13){$\Gamma$}
			
			\put(-85,11){$T\ije$}
			\put(-55,60){$D\ie^+$}
			\put(-48,-65){$D\je^-$}
			\put(-40,26){$S\ie^+$}
			\put(-92,-28){$S\je^-$}
			
			\put(-390,26){$\Omega^+$}
			\put(-390,-28){$\Omega^-$}
			
		\end{picture}
		\caption{
			The domain $\Omega$ (left) and the manifold $M\e$ (right). The dotted lines show how the passages $T\ije$ are glued to $M^\pm$ via identifications of some parts of their boundaries}
		\label{fig1}
	\end{figure}
	
	We start the construction of the sought manifold by defining the copies $M^+$ and $M^-$ of the domains $\Omega^+$ and $\Omega^-$, respectively (cf.~Remark~\ref{rem:placement} below). Namely, we set
	\begin{gather}\label{Mpm}
		M^\pm\ceq\Omega^\pm\pm e_n,\quad \Gamma^\pm\ceq \Gamma\pm e_n,\quad \text{where }e_n\ceq (0,\dots,0,1).
	\end{gather}
	We connect $M^+$ and $M^-$ by {a family} of passages. To do this, we first determine the subsets of $\Gamma^\pm$ to which these passages will be {attached}. Let $\eps>0$ be a small parameter, and
	$
	\{\mathbf{D}\ie^+ ,\ i\in \I\e \}
	$
	and
	$
	\{\mathbf{D}\ie^- ,\ i\in \I\e \}
	$
	be two {families} of Lipschitz domains in $\R^{n-1}$; here $\I\e$ is {some} finite set of indices. {On purpose} we do not specify the exact nature of the set $\I\e$, as {it} is completely irrelevant for the formulation and proof of {our} main results. In specific examples, { the choice $\I\e$ will follow from} the context. For instance, in the example we present in Section~\ref{sec:3},  $\I\e$ is a subset of $\Z^{n-1}\times\Z^{n-1}$. We {assume} that $\D\ie^\pm\subset\Ga $, within each of these families the domains are pairwise disjoint, and furthermore, there is a bijective map $\ell\e:\I\e\to \I\e$ such that
	\begin{gather}\label{leps}
		\mathbf{D}\ie^+\cong \mathbf{D}\je^-\text{ provided }j=\ell\e (i),
	\end{gather}
	{where $\cong$ is a congruence.} For $i\in \I\e$ we denote
	\begin{gather}\label{Die}
		D\ie^\pm\ceq\{x=(x^1,\dots,x^n)\in\R^n:\ (x^1,\dots,x^{n-1})\in \D^\pm\ie,\ x^n=\pm 1\}.
	\end{gather}
	Evidently, $D\ie^\pm\subset\Gamma^\pm$. We introduce the subset $\L\e$ of $\I\e\times \I\e$ by
	$$
	\L\e= \{(i,j):\ j=\ell\e(i)\}.
	$$
	Now, for {a given} $(i,j)\in \L\e$ we introduce the passage
	\begin{align}\notag
		T\ije
		&\ceq
		\{x\in\R^n: (x^1,\dots,x^{n-1})\in \D\ie^+ ,\, |x^n|\le h\ije/2 \}\\
		&\,\cong
		\{x\in\R^n: (x^1,\dots,x^{n-1})\in \D\je^- ,\, |x^n|\le h\ije/2 \},
		\label{Tije}
	\end{align}
	where $0<h\ije<1$; the congruence in \eqref{Tije} follows from \eqref{leps}. We denote by $S\ie^+$ and $S\je^-$ the top and bottom faces of $T\ije$, i.e.
	\begin{align*}
		S\ie^+&\ceq \{x\in\R^n: (x^1,\dots,x^{n-1})\in \D^+\ie,\, x^n=  h\ije/2\},\\
		S\je^-&\ceq \{x\in\R^n: (x^1,\dots,x^{n-1})\in \D^+\ie,\, x^n= - h\ije/2\},\ j=\ell\e(i).
	\end{align*}
	Finally, for $(i,j)\in \L\e$, we glue the set $T\ije$ to $M^+$ (respectively, $M^-$) by identifying $S^+\ie$ and $D^+\ie$ (respectively, $S^-\ie$ and $D^-\je$). By construction, the sets $S^+\ie$, $D^+\ie$, $S^-\je$, $D^-\je$ ($(i,j)\in \mathbb{L}\e$)
	are congruent, and the above identification is carrying out according to this congruence. {In this way, we obtain the} (topological) manifold
	\begin{gather}\label{Me}
		M\e=M^+\cup M^-\cup\Big[\cupl_{(i,j)\in \L\e}
		T\ije\Big]\Big/\sim,
	\end{gather}
	where $\sim$ denotes the above identification. The manifold $M\e$ is depicted on Figure~\ref{fig1} {(right)}. Its boundary has the form
	$$
	\partial M\e=
	\left[{\partial\Omega^+}\setminus \cupl_{i\in \I\e}D^+\ie\right]\cup
	\left[{\partial\Omega^-}\setminus \cupl_{i\in \I\e}D^-\ie\right]\cup
	\left[\cupl_{(i,j)\in \L\e}\left({\partial T\ije}\setminus (S\ie^+\cup S\je^-) \right)\right]
	$$
	
	\begin{remark}\label{rem:placement}
		As we will see below, the specific placement of the sets
		$M^\pm$ and $T_{ij}^\varepsilon $ in $\mathbb{R}^n$ plays no role:  instead of using the sets defined by \eqref{Mpm} and \eqref{Tije}, we may choose arbitrary congruent copies of them and connect these copies using the above  defined identifications.
	\end{remark}
	
	We equip $M\e$ with a metric simply assuming that is coincides with the Euclidean metric on each set $M^+$, $M^-$, $ { T}\ije$ constituting the whole manifold. Now, having the metric, we can introduce the Hilbert spaces $L^2(M\e)$ and $H^1(M\e)$; the latter consists of functions $u:M\e\to \C$ such that $u\restr_{M^\pm}\in H^1(M^\pm)$ and $u\restr_{T\ije}\in H^1(T\ije)$ (for each $(i,j)\in \L\e$), and the traces from both sides of $S\ie^+\sim D\ie^+$ and  the traces from both sides of $S\je^-\sim D\je^-$ {coincide}.
	
	In the following, the notation
	\begin{center}
		$\B_d(r,z)$ stands for the open ball in $\R^d$ of radius $r>0$ and center $z\in\R^d$;
	\end{center}
	it will be used later {with $d$ being either $n$ or $n-1$}. By $C,C_1,C_2,\dots$ we denote generic positive constants
	being independent of $\eps$; note that these constants may vary from line to line.
	As usual, the symbol $\rightharpoonup$ {means} the weak convergence in a Hilbert space.
	
	Now, we impose certain assumptions on the geometry of the manifold.
	We denote by $d\ie^+$ and $\x\ie^+$ (respectively, $d\ie^-$ and $\x\ie^-$) the radius and the center of the smallest ball $\B_{n-1}({d\ie^+},\x\ie^+)$ (respectively, $\B_{n-1}({d\ie^-},\x\ie^-)$\,) containing the set $\D\ie^+$ (respectively, $\D\ie^-$). The first block of assumptions concerns the shapes of the sets $\D\ie^\pm$.
	It is more convenient to formulate them in terms of the upscaled sets
	$$
	\wh{\D}^\pm \ie\ceq (d\ie^\pm)^{-1} \D\ie^\pm.
	$$
	Evidently, the outerradius of $\wh{\D}\ie$ equals $1$.
	In the following, the {symbol}
	$\Lambda_{\rm N}(\Omega)$ {means} the smallest non-zero eigenvalue of the Neumann Laplacian on
	a bounded connected open set $\Omega\subset\R^d$ with Lipschitz boundary.
	Our assumptions read as follows,
	\begin{itemize}
		\item The inradii of the sets $\wh{\D}^\pm\ie $ are bounded away from zero uniformly in $\eps $ and $i $:
		\begin{gather}\label{assump:shape1}
			\exists C>0\
			\forall \eps>0\ \forall i\in\I\e\
			\exists \mathbf{z}^\pm\ie\subset\R^{n-1}:\
			\B_{n-1}(C,\mathbf{z}^\pm\ie)\subset\wh{\D}^\pm\ie.
		\end{gather}
		\item The smallest non-zero eigenvalues of the Neumann Laplacian on $\wh\D\ie^\pm$ are bounded away from zero uniformly in $\eps$ and $i$:
		\begin{gather}\label{assump:shape2}
			\exists C>0\
			\forall \eps>0\ \forall i\in\I\e:\
			C\le\Lambda_{\rm N}(\wh\D\ie^\pm).
		\end{gather}
	\end{itemize}
	Condition \eqref{assump:shape1} is relatively straightforward --- it essentially prevents the use of very thin domains $\wh{\D}\ie$. In contrast, condition \eqref{assump:shape2} is more subtle. Apart from the trivial case where all $\wh{\D}\ie^\pm$ are similar, below we present two examples for which this assumption is fulfilled:
	\begin{itemize}
		\item If the sets $\wh{\D}\ie$ are \emph{convex}, then the assumption \eqref{assump:shape2} holds with $C= {\pi^2}/{4}$. This result immediately follows  from the Payne-Weinberger bound
		\footnote{{Cf. \cite{PW60};} is worth noting that the proof in {this paper} had a flaw for $n\ge 3$, which was later corrected in \cite{B03}.}
		$$
		\Lambda_{\rm N}(\Omega)\geq \frac{\pi^2}{(\diam  \Omega)^2},
		$$
		which is fulfilled for convex domains,
		combined with the fact that the outerradius of each $\wh{\bf D}\ie^\pm$ is equal to  $1$ {by assumption}.
		
		\item If $\wh{\D}\ie^\pm$ are smooth  \emph{strictly star-shaped domains} with respect some points $\mathbf{z}\ie^\pm$ (i.e., $h(\wh{\D}\ie^\pm)\ceq \min\limits_{\x\in\partial \wh{\D}\ie^\pm}\la  \x-\mathbf{z}\ie^\pm,\nu {(x)}\ra _{\R^{n-1}}>0$,
		where $\nu$ is the outward unit normal to $\partial \wh{\D}\ie^\pm$), furthermore, there are positive constants $C_1$, $C_2$ such that
		\begin{gather*}
			h(\wh{\D}\ie^\pm)\geq C_1,\quad
			\mathrm{dist}(\mathbf{z}\ie^\pm,\partial \wh{\D}\ie^\pm)\geq C_2,
		\end{gather*}
		then \eqref{assump:shape2}  holds with a constant $C$ depending on $C_1$, $C_2$ and the dimension $n$. This result follows easily from the bound established by Bramble and Payne \cite{BP62}:
		\begin{gather*}
			{\Lambda_{\rm N}}(\Omega)\ge {\frac{nR_{\min}(\Omega)^{n-1} h(\Omega)}
				{2(R_{\max}(\Omega))^2\big[(R_{\max}(\Omega))^n+ \frac{2}{n}R_{\min}(\Omega)^{n-1}h(\Omega)\big]}.}
		\end{gather*}
		Here $\Omega\subset\R^d$ is a strictly star-shaped domain with respect to the point $p$, $R_{\min}(\Omega)\ceq\min\limits_{x \in\partial \Omega} {\rm dist}(x,p)$,
		$R_{\max}(\Omega)\ceq\max\limits_{x \in\partial \Omega} {\rm dist}(x,p)$,
		$h(\Omega)\ceq \min\limits_{x\in\partial \Omega}\la  x-p,\nu {(x)}\ra _{\R^{d}}$,
		and $\nu$ is the outward unit normal to
		$\partial \Omega$.
	\end{itemize}
	
	The second block of assumptions addresses  those regarding the mutual arrangement of the sets $\D\ie^\pm$ and the sizes of $\D\ie^\pm$ and $T\ije$.
	We assume that for each $\eps>0$ there exist the numbers $ {\rho^\pm\ie\in (d^\pm\ie,1] }$, $i\in \I\e$ satisfying
	\begin{align}
		\label{assump:1}
		&\rho\e^\pm\ceq \sup_{i\in \I\e}\rho\ie^\pm\to 0\,\text{ as }\,\eps\to0,
		\\
		\label{assump:2}
		&
		\B_{n-1}({\rho^\pm\ie},\x^\pm\ie)\cap \B_{n-1}(\rho^\pm_{k,\eps},\x^\pm_{k,\eps}) =\emptyset,\; \forall i,k\in\I\e,\ i\not=k,
		\\[1mm]\label{assump:3}
		& \B_{n}({\rho^\pm\ie},x^\pm\ie)\cap\{x\in\R^n:\ \pm x^n>0\}\subset\Omega^\pm,\; \forall i\in\I\e,
		\\[1mm]
		&\label{assump:4}
		\sup_{i\in \I\e}(\rho\ie^\pm)^{n-1}q\ie^\pm\to 0\,\text{ as }\,\eps\to 0,\qquad  \text{cf.~Remark~\ref{rem:conditions}},
	\end{align}
	where  $x\ie^\pm\ceq(\x\ie^\pm,0)\in\Gamma$, and the quantities $q\ie^\pm$ are given by
	\begin{gather}\label{qie}
		q\ie^\pm\ceq
		\begin{cases}
			(d\ie^\pm)^{2-n},&n\ge 3,\\
			-\ln d\ie^\pm,&n=2.
		\end{cases}
	\end{gather}
	We also assume that the lengths of the passages $T\ije$ tend uniformly to zero, i.e.
	\begin{gather}
		\label{assump:5}
		h\e\ceq\sup_{(i,j)\in \L\e}h\ije\to 0\;\text{ as }\;\eps\to0.
	\end{gather}
	To formulate the last two assumptions, for $(i,j)\in \L\e$ we introduce {the} quantities $K\ije$  by
	\begin{gather}\label{Kij}
		K\ije\ceq \vol_{n-1}(\D\ie^+)h\ije^{-1}= \vol_{n-1}(\D_{j,\eps}^-)h\ije^{-1},\quad (i,j)\in \L\e
	\end{gather}
	(here $\vol_{n-1}(\cdot)$ means the volume of a domain in $\R^{n-1}$).
	Then our last assumptions read
	\begin{gather}\label{assump:main:1}
		\forall\eps>0\ \forall (i,j)\in  \L\e:\quad K\ije\leq C\min\left\{(\rho\ie^+)^{n-1},(\rho\je^-)^{n-1}\right\},\\\label{assump:main:2}
		\sum_{(i,j)\in \L\e} K\ije v(x\ie^+,x\je^-)\to\int_{\Gamma\times\Gamma}K(x,y)v(x,y)\d s_x \d s_y\,\text{ as }\,\eps\to0,\; \forall v\in C(\overline{\Gamma}\times\overline{\Gamma}),
	\end{gather}
	where $K\in C(\overline{\Gamma}\times\overline{\Gamma})$ is the function appearing in the definition of   $\H$.
	Note that
	\begin{align}\notag
		\forall i\in\I\e:\quad
		h\ije K\ije  (\rho\ie^+)^{1-n}&=\vol_{n-1}(\D\ie^+)(\rho\ie^+)^{1-n}=
		\vol_{n-1}(\wh \D\ie^+)
		(d\ie^+)^{n-1}(\rho\ie^+)^{1-n}\\&\geq
		C(d\ie^+)^{n-1}(\rho\ie^+)^{1-n},\text{ where }j=\ell\e(j).\label{hKest+}
	\end{align} 
	where {at} the last step we use \eqref{assump:shape1}.
	Similarly, one has
	\begin{align}\label{hKest-}
		\forall j\in\I\e:\quad
		h\ije K\ije  (\rho\je^-)^{1-n}\geq
		C(d\je^-)^{n-1}(\rho\je^-)^{1-n},\text{ where }i=\ell\e^{-1}(j).
	\end{align} 
	From \eqref{assump:5}, \eqref{assump:main:1}, \eqref{hKest+}, \eqref{hKest-} we deduce
	\begin{gather}\label{drho}
		\sup_{i\in\I\e} d\ie^\pm (\rho\ie^\pm)^{-1}\leq \sqrt[n-1]{h\e}\, \to 0\,\text{ as }\,\eps\to0.
	\end{gather}
	
	In Section~\ref{sec:3},  we {will} present an example where all the aforementioned properties hold.
	
	\begin{remark}\label{rem:conditions}
		If the assumptions \eqref{assump:1}--\eqref{assump:3} hold with some $\rho\ie^\pm=\wt\rho\ie^\pm$, then they obviously holds with any $\rho\ie^{\pm}$  satisfying
		$d\ie^\pm<\rho\ie^\pm<\wt\rho\ie^\pm$. It is easy to see that  we can always choose $\rho\ie^\pm\in (d\ie^\pm,\wt\rho\ie^\pm)$ to be sufficiently small  in order to fulfill \eqref{assump:4} (even if \eqref{assump:4} does not hold with $\wt\rho\ie^\pm$). {Hence at a} glance, this might suggest that the assumption \eqref{assump:4} is redundant. However, this is not the case when we take into account the last two assumptions \eqref{assump:main:1}--\eqref{assump:main:2}. Indeed, using \eqref{assump:main:1}, we get the following estimate for the left-hand-side of \eqref{assump:main:2}:
		\begin{gather*}
			\bigg|\sum_{(i,j)\in \L\e} K\ije v(x\ie^+,x\je^-)\bigg| \le C\sum_{(i,j)\in\L\e}|K\ije| \leq
			C_1\sum_{i\in\I\e}( \rho\ie^\pm)^{n-1}\le
			C_2\sum_{i\in\I\e}\vol_{n-1}(\B_{n-1}( \rho\ie^\pm,\x\ie^\pm)).
		\end{gather*}
		Thus, to ensure that the limiting function $K(x,y)$ is nonzero, we cannot allow $\rho\ie^\pm$ to become small --- the total volume of the pairwise disjoint balls $\B_{n-1}(\rho\ie^\pm,\x\ie^\pm)$ must not tend to zero.
	\end{remark}
	
	We denote by $\H\e$ the Neumann Laplacian in $M\e$, i.e. the operator acting in   $L^2(M\e)$  being associated with the sesquilinear form
	\begin{align*}
		\h\e[u,v]&=
		(\nabla u,\nabla v)_{L^2(M\e)}\\
		&=
		\ds\int_{M^+}
		\suml_{k=1}^n \frac{\partial u}{\partial x^k} \frac{\partial \overline{ v}}{\partial x^k} \d x+\int_{M^-}
		\suml_{k=1}^n \frac{\partial u}{\partial x^k} \frac{\partial \overline{ v}}{\partial x^k} \d x+\sum_{(i,j)\in \L\e}\int_{T\ije}
		\suml_{k=1}^n \frac{\partial u}{\partial x^k} \frac{\partial \overline{ v}}{\partial x^k} \d x,
	\end{align*}
	on the domain $\dom(\h\e)= H^1(M\e)$.
	
	Since the operators $\H\e$ and $\H$ act in different Hilbert spaces $L^2(M\e)$
	and $L^2(\Omega)$, we need a suitable identification operator $\J\e:L^2(M\e)\to L^2(\Omega)$.
	The natural choice for this operator is as follows,
	\begin{gather}\label{Je}
		(\J\e f)(x^1,\dots,x^{n-1},x^n)=
		\begin{cases}
			f(x^1,\dots,x^{n-1},x^n+1),&  x=(x^1,\dots,x^n)\in\Omega^+ ,\\
			f(x^1,\dots,x^{n-1},x^n-1),&  x=(x^1,\dots,x^n)\in\Omega^- .
		\end{cases}
	\end{gather}
	The dependence of this operator on $\eps$ is merely formal --
	it just `drags' $f$ from $M^+\cup M^-$ to $\Omega^+\cup \Omega^-$,
	while the values of $f$ on $T\ije$ do not affect $\J\e f$.
	Note that $\J\e f\in H^1(\Omega\setminus\Gamma)$ provided $f\in H^1(M\e)$.
	
	We are now in position to formulate the main results.
	
	\begin{theorem}\label{th1}
		Assume that the conditions 
		\eqref{assump:shape1}--\eqref{assump:4}, \eqref{assump:5},
		\eqref{assump:main:1}, \eqref{assump:main:2} hold true.
		Let $\{f\e\in L^2(M\e)\}_{\eps}$ be a family of functions satisfying
		\begin{gather}\label{assump:f1}
			\lim_{\eps\to 0}\sum_{(i,j)\in \L\e}\|f\|^2_{L^2(T\ije)}=0,\\\label{assump:f2}
			\J\e f\e \rightharpoonup  f\text{ in }L^2(\Omega)\,\text{ as }\,\eps\to 0,\; f\in L^2(\Omega).
		\end{gather}
		Then, the following holds:
		\begin{gather}\label{th1:conv}
			\J\e{(\H\e +\Id )^{-1}}f\e \rightharpoonup {(\H+\Id )^{-1}}f\,\text{ in }\,H^1(\Omega\setminus\Gamma)\,\text{ as }\,\eps\to0.
		\end{gather}
	\end{theorem}
	
	\begin{remark}
		By Rellich-Kondrashev theorem, \eqref{th1:conv} implies
		\begin{gather*}
			\J\e{(\H\e +\Id )^{-1}}f\e \to {(\H+\Id )^{-1}}f\,\text{ in }\,L^2(\Omega)\,\text{ as }\,\eps\to0.
		\end{gather*}
		Hence, Theorem~\ref{th1} establishes a kind of strong resolvent convergence of the operators $\H\e$ to operator $\H$ as $\eps\to0$.
	\end{remark}
	The second result addresses the convergence of spectra.
	Since the manifold $M\e$ is compact, the spectrum of the operator $\H$ is purely discrete. We denote by $\{\lambda\ke,\ k\in\N\}$ the sequence of eigenvalues of the operator $\H\e$ ordered in ascending order and counted with multiplicities. We denote by $\{u\ke,\ k\in\N\}$ the associated sequence of eigenfunctions such that $(u_{k,\eps},u_{\ell,\eps})_{L^2(M\e)}=\delta_{k\ell}$. Similarly, by $\{\lambda_k ,\ k\in\N\}$ we denote the sequence of eigenvalues of the operator $\H$.
	\begin{theorem}\label{th2}
		Assume that the conditions 
		\eqref{assump:shape1}--\eqref{assump:4}, \eqref{assump:5},
		\eqref{assump:main:1}, \eqref{assump:main:2} hold true. Then one has
		\begin{gather}\label{th2:conv1}
			\forall k\in\N:\ \lambda\ke\to\lambda_k\,\text{ as }\,\eps\to0.
		\end{gather}
		Furthermore, if the eigenvalue $\lambda_j$ satisfies
		\begin{gather}	\label{lambdaj}
			\lambda_{j-1}<\lambda_j=\lambda_{j+1}=\dots=\lambda_{j+m-1}<\lambda_{j+m}
		\end{gather}
		(i.e., $\lambda_j$ has multiplicity $m$) and $u$ is an eigenfunction of $\H$ corresponding to $\lambda_j$, then there exists a sequence
		$u\e\in\mathrm{span}\{u_{k,\eps},\ k=j,\dots,j+m-1\}$ such that
		\begin{gather}\label{th2:conv2}
			\|\J\e u\e-u\|_{L^2(\Omega)}\to 0.
		\end{gather}
	\end{theorem}
	
	\medskip
	
	Finally, we {shall} examine the behavior of solutions $v\e(t)={\exp(-\H\e t)f\e}$ to the {associated} Cauchy problem
	\begin{gather*}
		{\ds{\partial v\e\over \partial t}+\H\e v\e = 0}, \;\;\ t>0,\quad
		v\e(0)=f\e.
	\end{gather*}
	\begin{theorem}\label{th3}
		Assume that the conditions 
		\eqref{assump:shape1}--\eqref{assump:4}, \eqref{assump:5},
		\eqref{assump:main:1}, \eqref{assump:main:2} hold true.
		Let $\{f\e\in H^1(M\e)\}_{\eps}$ be a family of functions satisfying the assumptions \eqref{assump:f1}, \eqref{assump:f2}, moreover
		\begin{gather}\label{fe:bound}
			\|f\e\|_{H^1(M\e)}\leq C.
		\end{gather}
		Then the following holds:
		\begin{gather*}
			\forall T>0:\quad \lim_{\eps\to 0} \max_{t\in[0,T]}
			{\big\|\J\e(\exp(-\H\e t)f\e)
				-\exp(-\H t)f\big\|_{L^2(\Omega)}}=0.
		\end{gather*}
	\end{theorem}

	\section{Example}\label{sec:3}
	
	In this section, we present an example of a manifold that satisfies all the assumptions required for the main theorems. {We begin with} some preliminary considerations. For $s=(s^1,\dots,s^{n-1})\in\Z^{n-1}$ we denote by $\Ga_{s,\eps}$ the cube in $\R^{n-1}$ with a side length $\eps$ and the center at $\eps s$, being oriented along the coordinate axes:
	\begin{gather*}
		\Ga_{s,\eps}=\big\{\x=(x^1\dots,x^{n-1})\in\R^{n-1}:\ |x^k-\eps s^k|<\eps/2,\ k=1,\dots,n-1\big\}.
	\end{gather*}
	By $\Z\e\subset \Z^{n-1}$ we denote the set of  indices $s\in\Z^{n-1}$  such that 
	\begin{gather}\label{eps:neighb}
		\{x\in\R^n:\ (x^1,\dots,x^{n-1})\subset \Ga_{s,\eps},\ x^n\in (-\eps,\eps)\}\subset \Omega. 
	\end{gather}
	Next, within each set {$\Ga_{s,\eps}$} we determine another $\# \Z\e$ {cubes} $\mathbf \Gamma_{s,t,\eps}$ {with} $t\in\Z\e$:
	\begin{gather*}
		\forall s,t\in \Z\e:\
		\Ga_{s,t,\eps}\ceq\eps L^{-1}\mathbf \Gamma_{t,\eps}+\eps s,
	\end{gather*}
	where $L$ is the {edge} length of the smallest cube in $\R^{n-1}$ oriented along the coordinate axes and containing {the set} $\Ga $; without loss of generality, we assume that this smallest cube is centered at the origin.
	Each cube $\Ga_{s,t,\eps}$ has {edge} length $L^{-1}\eps^2$ and is centered at the point $\x_{s,t,\eps}\ceq    \eps^2 L^{-1}t+\eps s$; the presence of the factor $L^{-1}$ guarantees that $\cup_{t\in \Z\e}\Ga_{s,t,\eps}\subset\Ga_s$. Finally, let $\D_{s,t,\eps}$ with $(s,t)\in\Z\e\times\Z\e$ be the ball of the radius $\al_{s,t,\eps} d\e$ centered at $\x_{s,t,\eps}$ with $d\e>0$ satisfying
	\begin{gather}\label{assump:5+}
		\lim_{\eps\to 0}\, d\e\eps^{-2}=0,\\ \label{assump:4+}
		\lim_{\eps\to 0}\,\eps^{2(n-1)} d\e^{2-n}=0\,\text{ for }\,n\ge 3,
		\qquad
		\lim_{\eps\to 0}\,\eps^2\ln d\e=0\,\text{ for }\,n=2,
	\end{gather}
	and {with} $\al\ije>0$ being defined by
	\begin{gather}\label{al:ije}
		\al_{s,t,\eps}=(K( x_{s,t,\eps}, x_{t,s,\eps}))^{1/(n-1)},\text{ where }
		x_{s,t,\eps}\ceq(\x_{s,t,\eps},0)\in\Gamma
	\end{gather}
	(recall that $K\in C(\overline{\Gamma}\times\overline{\Gamma})$ is {the} function {appearing} in the definition of operator $\H$). {In view of the symmetry} \eqref{K:prop} {we have} $\al_{s,t,\eps}=\al_{t,s,\eps}$, so the balls $\D_{s,t,\eps}$ and $\D_{t,s,\eps}$ have equal radii.
	
	\begin{figure}[ht]
		\centering
		\includegraphics[width=\textwidth]{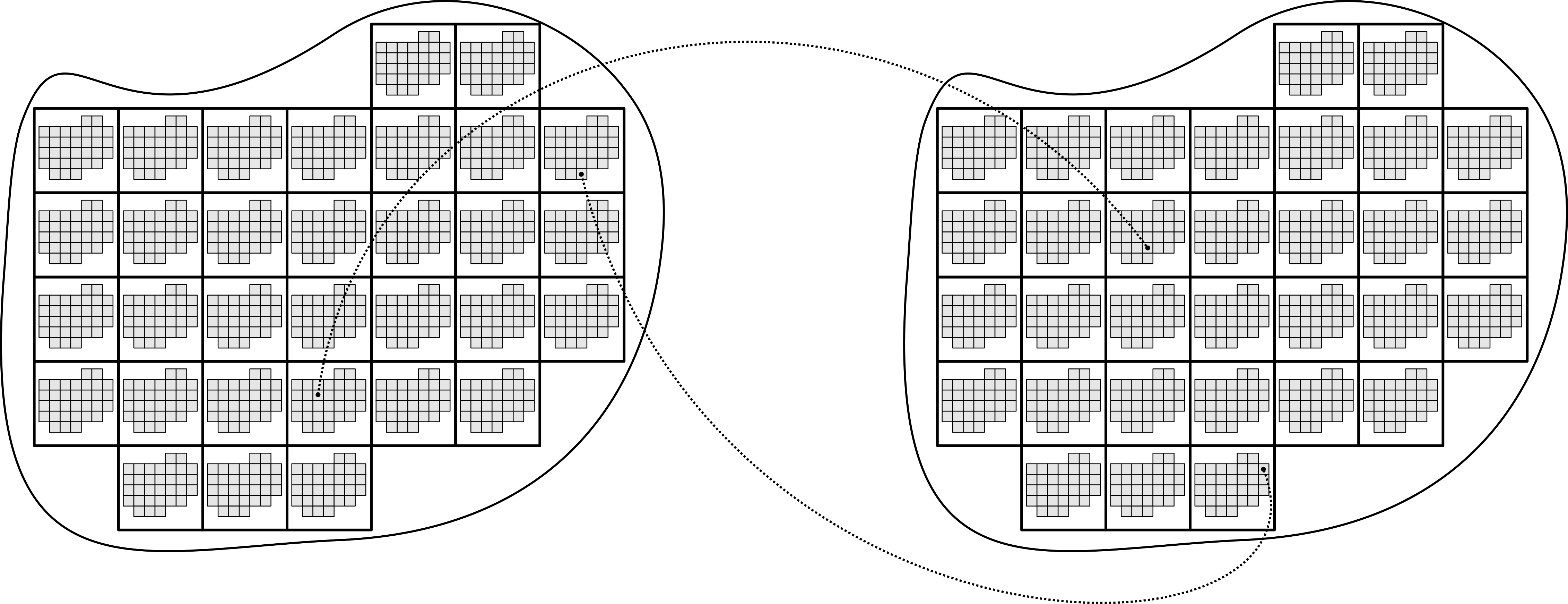}
		\caption{The sets $\Gamma^+$ (left) and $\Gamma^-$ (right) {and their cube subsets}; we present them side by side, rather than stacked vertically, to provide a clearer view. The big cubes correspond to the sets $\Ga_{s,\eps}\times\{1\}$ (left) and  $\Ga_{s,\eps}\times\{-1\}$ (right), while the small {gray} cubes correspond to  $\Ga_{s,t,\eps}\times\{1\}$ (left) and  $\Ga_{s,t,\eps}\times\{-1\}$ (right). {The set $D_{i,\eps}^\pm$, $\,i=(s,t)$, is placed in the center of the cube} $\Ga_{s,t,\eps}\times\{\pm1\}$  (we choose not to draw {the balls} to avoid overloading the figure). The dotted lines represent pairs of small squares such that the sets $D_{i,\eps}^+$, $i=(s,t)$  and $D_{j,\eps}^-$, $j=(t,s)$, located within these squares, are connected by the passage $T\ije$.}
	\end{figure}
	
	With these preliminaries, we can now {specify} the sets $\D\ie^\pm$ appearing in the definition of the manifold {$M\e$}. {To begin with, we introduce the index set} $\I\e$ as a subset of $\Z^{n-1}\times\Z^{n-1}$:
	$$
	\I\e\ceq \Z\e\times\Z\e.
	$$
	Then {we put}
	$$
	\D\ie^\pm\ceq \D_{s,t,\eps}\text{ for any $i=(s,t)\in\I\e$},
	$$
	and define the sets $D\ie^\pm$ by \eqref{Die}. 
	Since the sets $\D\ie^\pm$ are balls, we immediately obtain 
	\begin{gather}\label{xst}
		\x\ie^\pm=\x_{s,t,\eps},\ x\ie^\pm=x_{s,t,\eps},\ i=(s,t)\in \I\e.
	\end{gather}     
	We {introduce}   $\ell\e:\I\e\to \I\e$ as follows:
	\begin{gather*}
		\ell\e\text{ maps }i=(s,t)\text{ with }s,t\in\Z\e\text{ to }j=(t,s).
	\end{gather*}
	Evidently, $\ell\e$ is a bijective map.
	Let the passage $T\ije$ {with} $i\in\I\e$ {and} $j=\ell\e(i)$ be given by \eqref{Tije} with
	\begin{gather}\label{h:ije}
		h\ije\ceq {\vol_{n-1}(\B_{n-1}(1,0))}d\e^{n-1}{\eps^{2(1-n)}}.
	\end{gather}
	Using the map $\ell\e$, we glue each set $T\ije$ to $M^+$ (respectively, $M^-$) by identifying the upper face of $T\ije$ {with} $D^+\ie$ (respectively, the bottom face of $T\ije$ {with} $D^-_{j,\eps}$); {in this way} we arrive at the manifold $M\e$ defined by \eqref{Me}.
	
	Since the sets $\D\ie^\pm$ are balls, the corresponding {upscaled} sets $\wh \D\ie^\pm$ are unit balls obviously {satisfying} the assumptions \eqref{assump:shape1} and \eqref{assump:shape2}. Evidently, {relations} \eqref{assump:1}--\eqref{assump:3} hold true with
	\begin{gather}
		\label{rhoL}\rho\ie^\pm=L^{-1}\eps^2/2
	\end{gather}
	(in particular, the fulfillment of \eqref{assump:3} follows from \eqref{eps:neighb}).
	{In view of} \eqref{assump:4+}, the assumption \eqref{assump:4} {is satisfied} as well, {and the} assumption \eqref{assump:5} is fulfilled due to \eqref{assump:5+}. {By} simple manipulations 
	(see \eqref{Kij}, \eqref{al:ije}--\eqref{h:ije}), we obtain
	$$
	K\ije=K(  x_{s,t,\eps}, x_{t,s,\eps})\eps^{2(n-1)},\ i=(s,t)\in\I\e,\ j=(t,s)=\ell\e(i),
	$$
	so the assumption \eqref{assump:main:1} {is also fulfilled} (cf.~\eqref{rhoL}). Finally, we {shall} verify \eqref{assump:main:2}. Let $v\in C(\overline{\Gamma}\times\overline{\Gamma})$.
	By direct calculations we get
	\begin{align}\notag
		\sum_{(i,j)\in \L\e} K\ije v(x\ie^+,x\je^-)&=\sum_{(s,t)\in\Z\e\times\Z\e}K(   x_{s,t,\eps},  x_{t,s,\eps})v(  x_{s,t,\eps}, x_{t,s,\eps})\eps^{2(n-1)}.
		\\\label{Ke:action}
		&=\sum_{(s,t)\in\Z\e\times\Z\e}\int_{\Gamma_{s,\eps}\times\Gamma_{t,\eps}}K(   x_{s,t,\eps},  x_{t,s,\eps})v(  x_{s,t,\eps}, x_{t,s,\eps})\d s_x \d s_y,
	\end{align}
	where $\Gamma_{s,\eps}\ceq\{x\in\R^n:\ (x^1,\dots,x^{n-1})\in \Ga_{s,\eps},\ x^n=0\}$. 
	By construction, we have
	$$(x_{s,t,\eps},x_{t,s,\eps})\in \Gamma_{s,\eps}\times\Gamma_{t,\eps},\ \forall (s,t)\in \Z\e\times\Z\e,$$
	the sets $\{\Gamma_{s,\eps}\times \Gamma_{s,\eps},\ (s,t)\in\Z\e\times\Z\e\}$ are mutually disjoint, and one has
	$$\bigcup\limits_{(s,t)\in \Z\e\times\Z\e}(\Gamma_{s,\eps}\times\Gamma_{t,\eps})\subset \Gamma\times\Gamma\text{\quad and\quad }
	\int_{(\Gamma\times\Gamma)\setminus\cup_{(s,t)\in \Z\e\times\Z\e}(\Gamma_{s,\eps}\times\Gamma_{t,\eps})}\d s_x \d s_y\to 0\text{ as }\eps\to 0.$$   From these facts we easily  conclude that the right-hand-side of {relation} \eqref{Ke:action} converges to  the integral over $\Gamma\times\Gamma$ of the function $K(x,y)v(x,y)$ that implies the fulfillment of 
	the last assumption \eqref{assump:main:2}.

	\section{Proof of the main results}\label{sec:4}

	In {what follows}, we assume that $\eps$ is sufficiently small {to ensure that}
	\begin{gather}\label{2drho}
		2d^\pm\ie\leq \rho\ie^\pm, 
	\end{gather}
	cf.~\eqref{drho}. Recall that $x\ie^\pm=(\x\ie^\pm,0)$, $e_n=(0,\dots,0,1)$, $\Gamma^\pm=\Gamma\pm e_n$. We denote $$\wt x\ie^\pm\ceq (\x\ie^\pm,\pm1)=x\ie^\pm\pm e_n\in\Gamma^\pm,$$ and introduce the sets
	\begin{align}
		\label{Yie}
		Y\ie^\pm&\ceq \B_{n}(\rho\ie^\pm,\wt x\ie^\pm)\cap 
		\{x \in \R^n:\ \pm x^n >1 \},\\\notag
		\wt Y\ie^\pm&\ceq \B_{n}(2 d\ie^\pm,\wt x\ie^\pm)\cap \{x \in \R^n:\ \pm x^n >1 \},
		\\\notag
		\Sigma\ie^\pm&\ceq \partial Y\ie^{\pm} \setminus \Gamma^{\pm},
		\\\notag
		\wt \Sigma\ie^\pm&\ceq \partial \wt Y\ie^{\pm} \setminus \Gamma^\pm,\\\notag
		B\ie^\pm&\ceq \partial Y\ie\cap \Gamma^\pm,
		\\\notag
		F\ie^\pm&\ceq \{x\in \R^n:\ (x^1,\dots,x^{n-1})\in \D\ie^\pm,\ \pm  x^n\in (1,1+d^\pm\ie)\}.
	\end{align}
	One has $F\ie^\pm\subset \wt Y\ie^\pm\subset Y\ie^\pm\subset M^\pm$, 
	{with  the first inclusion following from the definitions of the involved sets, the second inclusion being a consequence of \eqref{2drho}
		and the third one following from \eqref{assump:3}}. Also, we have $D\ie^\pm\subset B\ie^\pm$.

	In the following, by $\vol_d(D)$ we denote the volume of
	a domain $D\subset\R^d$. This notation will be used for $d=n$ and $d=n-1$, e.g., $\vol_{n}(Y\ie^\pm)$, $\vol_{n-1}(\D\ie^\pm)$.
	Also, if $S$ is a subset
	of an $(n-1)$-dimensional hypersurface in $\R^n$, we denote by
	$\area_{n-1}(S)$ its area, i.e. $\area_{n-1}(S)=\int_S \d s_x$, where
	$\d s_x$ is the density of the surface measure on $S$.
	These notations can be  demonstrated with the following example identities: $\area_{n-1}(D^\pm\ie)=\vol_{n-1}(\D^\pm\ie) $, 
	$\area_{n-1}(\Gamma)=\vol_{n-1}(\Ga) $.
	
	{By} $\la f\ra _{\Omega}$ we denote the mean value of a function $f$ in  an open bounded set $\Omega\subset\R^n$:
	\begin{equation*}
		\la f\ra_{\Omega}=(\vol_n(\Omega))^{-1}\int_{\Omega}f(x)\d x,
	\end{equation*}
	The same notation is  used for the mean value of a function $f$ on a
	subset $S$ of a $(n-1)$-dimensional hypersurface in $\R^{n-1}$, i.e.
	\begin{equation*}
		\la f\ra_{S}=(\area_{n-1}(S))^{-1}\int_{S}f(x)\d s_x.
	\end{equation*}
	
	\subsection{Auxiliary estimates}
	
	In this section we collect several functional estimates. 
	The estimate in Lemma~\ref{lemma:main}, which will be established through a series of intermediate results in Lemmata~\ref{lemma:aux1}-\ref{lemma:aux5}, plays an important role in the proof of Theorem~\ref{th1}.
	Lemma~\ref{lemma:main2} will be employed in the proof of Theorem~\ref{th2}.
	\begin{lemma}\label{lemma:aux1}
		{There is a $C>0$ such that}
		\begin{equation}\label{lemma:aux1:est}	
			\forall\, f\in H^1(F\ie^\pm):\quad
			\big|\la  f \ra _{D\ie^\pm} - \la  f\ra _{F\ie^\pm}\big|^2 \le
			C (d\ie^\pm)^{2-n} \|\nabla f\|^2_{L^2(F\ie^\pm)}.
		\end{equation}	
	\end{lemma}
	\begin{proof}
		We will establish \eqref{lemma:aux1:est} for $F\ie^+$ {only;} for $F\ie^-$  the {the argument can be} repeated \emph{verbatim}. Let $u\in C^1(\overline{F\ie^\pm})$. For $x=(x^1,\dots,x^{n-1},1)$, $z=(x^1,\dots,x^{n-1},t) $ with $(x^1,\dots,x^{n-1})\in\D\ie^+ $ and $t\in (1,1+d\ie^+)$
		({thus} $x\in D\ie^+$ and $z\in F\ie^+$), we have
		$$
		u(x)=u(z)-\int_{1}^t {\partial u\over\partial x^n}(x^1,\dots,x^{n-1},\tau)\d \tau,
		$$
		whence
		\begin{align}\notag
			|u(x)|^2&\le 2|u(z)|^2+2(t-1)\int_{1}^t \left|{\partial u\over\partial x^n}(x^1,\dots,x^{n-1},\tau)\right|^2\d \tau\\
			&\le 2|u(z)|^2+2d\ie^+\int_{1}^{1+d\ie^+} \left|{\partial u\over\partial x^n}(x^1,\dots,x^{n-1},\tau)\right|^2\d \tau.\label{trace:Fie+}
		\end{align}
		Integrating {relation} \eqref{trace:Fie+} over {the cylindrical sets $F\ie^\pm$, i.e. over the variables} $t\in ({1},{1+d\ie^+})$ and $(x^1,\dots,x^{n-1})\in \D\ie^+$, and dividing {the result} by $d\ie^+$, we get
		\begin{gather}\label{trace:Fie}
			\|u\|^2_{L^2(D\ie^+)}
			\le 2(d\ie^+)^{-1}\|u\|^2_{L^2(F\ie^+)}+2d\ie^+\|\nabla u\|^2_{L^2(F\ie^+)},
		\end{gather}
		{and by} the density argument, the estimate \eqref{trace:Fie} holds for any $u\in H^1(F\ie^+)$.
		
		{Next} we apply \eqref{trace:Fie} {to} $u=f-\la f \ra_{F\ie}$. 
		Using the Cauchy-Schwarz inequality, we obtain:
		\begin{align}\notag
			\big|\la  f \ra _{D\ie^+} - \la  f\ra _{F\ie^+}\big|^2&=
			\big|\la  f   - \la  f\ra _{F\ie^+}\ra_{D\ie^+}\big|^2\leq
			(\area_{n-1}(D\ie^+))^{-1}\|f   - \la  f\ra _{F\ie^+}\|^2_{L^2(D\ie^+)}\\\notag
			&\leq
			2\,(\area_{n-1}(D\ie^+))^{-1}\left((d\ie^+)^{-1}\|f   - \la  f\ra _{F\ie^+}\|^2_{L^2(F\ie^+)}+
			d\ie^+\|\nabla f   \|^2_{L^2(F\ie^+)}\right)
			\\
			&\leq
			2\,(\area_{n-1}(D\ie^+))^{-1}\big((\Lambda_N(F\ie^+)d\ie^+)^{-1}+d\ie^+\big)\|\nabla f   \|^2_{L^2(F\ie^+)},\label{lm1:1}
		\end{align}
		where in the last step we have used the standard Poincar\'e inequality
		\begin{gather}\label{Poincare}
			\|u-\la u\ra_\Omega\|^2_{L^2(\Omega)} \leq (\Lambda_N(\Omega))^{-1}\|\nabla u\|^2_{L^2(\Omega)},
		\end{gather}
		which holds for an arbitrary bounded Lipschitz domain $\Omega\subset\R^d$; recall that the {symbol} $\Lambda_N(\Omega)$ stands for the first non-zero eigenvalue of the Neumann Laplacian on $\Omega$. Using \eqref{assump:shape1}, we deduce
		\begin{gather}\label{lm1:2}
			{  {\area_{n-1}(D\ie^+)}={\mathrm{vol}_{n-1}(\D\ie^+)}}
			=(d\ie^+)^{n-1}{\mathrm{vol}_{n-1}(\wh\D\ie^+)}
			\geq C(d\ie^+)^{n-1}.
		\end{gather}
		Furthermore, since $F\ie^+$ is a cylinder with the base $D\ie^+$ and the height $d\ie^+$, we have
		\begin{gather}\label{lm1:3}
			\Lambda_N(F\ie^+)=\min\left\{\Lambda_N(\D\ie^+),\, ( {\pi}/{d\ie^+} )^2\right\}.
		\end{gather}
		Using simple scaling arguments and \eqref{assump:shape2}, we obtain
		\begin{gather}\label{lm1:4}
			\Lambda_N(\D\ie^+)=(d\ie^+)^{-2}\Lambda_N(\wh\D\ie^+)\geq  C(d\ie^+)^{-2}.
		\end{gather}
		Combining {finally} \eqref{lm1:1}, \eqref{lm1:2}--\eqref{lm1:4}, we arrive at the {sought} estimate \eqref{lemma:aux1:est}.
	\end{proof}
	
	\begin{lemma}\label{lemma:aux2}
		{There is a $C>0$ such that}
		\begin{equation} 	\label{lemma:aux2:est}
			\forall\, f\in H^1(\wt Y\ie^\pm):\quad
			\big|\la  f \ra _{F\ie^\pm} - \la  f\ra _{\wt Y\ie^\pm}\big|^2 \le
			C (d\ie^\pm)^{2-n} \|\nabla f\|^2_{L^2(\wt Y\ie^\pm)}.
		\end{equation}
	\end{lemma}
	\begin{proof}
		Taking into account   $F\ie^\pm\subset \wt Y\ie^\pm$ and using the
		Cauchy-Schwarz inequality and the
		Poincar\'e inequality \eqref{Poincare}, we get
		\begin{align}\notag
			\big|\la  f \ra _{F\ie^\pm} - \la  f\ra _{\wt Y\ie^\pm}\big|^2&=
			\big|\la  f - \la  f\ra _{\wt Y\ie^\pm} \ra _{F\ie^\pm}\big|^2
			\leq
			(\vol_{n}(F\ie^\pm))^{-1}\| f - \la  f\ra _{\wt Y\ie^\pm}\|^2_{L^2(F\ie^\pm)}
			\\\notag
			&\leq
			(\vol_{n}(F\ie^\pm))^{-1}\| f - \la  f\ra _{\wt Y\ie^\pm}\|^2_{L^2(\wt Y\ie^\pm)}
			\\\label{lm2:1}
			&\leq
			(\vol_{n}(F\ie^\pm)\cdot \Lambda_N(\wt Y\ie^\pm))^{-1}\|\nabla f \|^2_{L^2(\wt Y\ie^\pm)}.
		\end{align}
		Evidently, one has
		\begin{gather}\label{lm2:2}
			\Lambda_N(\wt Y\ie^\pm)=C\,(2d\ie^+)^{-2},
		\end{gather}
		where the contant $C$ is the first non-zero eigenvalue of the Neumann Laplacian on a unit half-ball in $\R^n$. {Furthermore, in view of} \eqref{assump:shape1} we have
		\begin{gather}\label{lm2:3}
			\vol_{n}(F\ie^\pm)=d\ie^\pm{\mathrm{vol}_{n-1}(\D\ie^\pm)}=(d\ie^\pm)^{n }{\mathrm{vol}_{n-1}(\wh\D\ie^\pm)}\ge C (d\ie^\pm)^{n } ;
		\end{gather}
		{the} estimate \eqref{lemma:aux2:est} {then} follows from \eqref{lm2:1}--\eqref{lm2:3}.
	\end{proof}
	
	\begin{lemma}\label{lemma:aux3}
		{There is a $C>0$ such that}
		\begin{equation} 	\label{lemma:aux3:est}
			\forall\, f\in H^1(\wt Y\ie^\pm):\quad
			\big|\la  f \ra _{\wt Y\ie^\pm} - \la  f\ra _{\wt \Sigma\ie^\pm}\big|^2 \le
			C (d\ie^\pm)^{2-n} \|\nabla f\|^2_{L^2(\wt Y\ie^\pm)}.
		\end{equation}	
	\end{lemma}
	\begin{proof}
		We introduce the sets
		$$
		\wt Y^\pm\ceq (2d\ie^\pm)^{-1} (\wt Y\ie^\pm - \wt x\ie^\pm),\quad
		\wt \Sigma^\pm\ceq (2d\ie^\pm)^{-1} (\wt \Sigma\ie^\pm - \wt x\ie^\pm).
		$$
		Evidently, $\wt   Y^\pm$ are unit half-balls in $\R^n$, $\wt \Sigma^\pm\subset \partial \wt  Y^\pm$ are unit half-spheres, {both referring to a ball and sphere, respectively, centered at the origin. We employ} the trace estimate
		\begin{gather*}
			\forall u\in H^1(\Omega):\ \|u\|^2_{L^2(\Sigma)}\leq C  \|u\|^2_{H^1(\Omega)},
		\end{gather*}
		where $\Omega\subset\R^d$ is a Lipschitz domain, $\Sigma\subset\partial\Omega$ {and} the constant $C  $ depends on $\Omega$ and $\Sigma$ {only. Applying it to} $u=f-\la f \ra_{\wt\Sigma^\pm}$ and using the Cauchy-Schwarz inequality, we obtain
		\begin{align}\notag
			\big|\la  f \ra _{\wt\Sigma^\pm} - \la  f\ra _{\wt Y^\pm}\big|^2&=
			\big|\la  f  - \la  f\ra _{\wt Y^\pm}\ra _{\wt \Sigma^\pm}\big|^2\leq
			(\area_{n-1}(\wt\Sigma^\pm))^{-1}\|f-\la  f\ra _{\wt Y^\pm}\|^2_{L^2(\wt\Sigma^\pm)}\\
			\label{lm3:1}
			& \leq
			C\|f-\la  f\ra _{\wt Y^\pm}\|^2_{H^1(\wt Y^\pm)}\leq
			C_1\|\nabla f\|^2_{L^2(\wt Y^\pm)},
		\end{align}
		where {in} the last step we also use the Poincar\'e inequality \eqref{Poincare}. Using the {shift-and-scaling} coordinate transformation, $ x\mapsto y=2d\ie^\pm\cdot x+\wt x\ie^\pm$, {which maps} $\wt Y^\pm$ onto $\wt Y^\pm\ie$, we reduce the validity of the sought estimate \eqref{lemma:aux3:est} to that of \eqref{lm3:1}.
	\end{proof}
	
	\begin{lemma}\label{lemma:aux4}
		{There is a $C>0$ such that}
		\begin{equation} 	\label{lemma:aux4:est}
			\forall\, f\in H^1(R\ie^\pm):\quad
			\big|\la  f \ra _{\wt \Sigma\ie^\pm} - \la  f\ra _{  \Sigma\ie^\pm}\big|^2 \le
			C q\ie^\pm \|\nabla f\|^2_{L^2(Y\ie^\pm\setminus\wt Y\ie^\pm)}.
		\end{equation}	
	\end{lemma}
	\begin{proof}
		{Recall that $q\ie^\pm$ is defined by \eqref{qie}. It is evidently} enough to prove \eqref{lemma:aux4:est} for $f\in C^1(\overline{Y\ie^\pm\setminus\wt Y\ie^\pm})$. We introduce polar coordinate system $(r,\phi)$ with the pole at $\wt x\ie^\pm$; here $r>0$ stands for the distance to the pole and $\phi=(\phi_1,\dots,\phi_{n-1})$ are the angular coordinates. One has
		\begin{align}\notag
			{\la  f \ra _{\wt \Sigma\ie^\pm} - \la  f\ra _{  \Sigma\ie^\pm}}&=
			(\area_{n-1}(\partial \B_{n}(1,0)))^{-1}
			\left(\int_{\partial \B_{n}(1,0)}f({2d\ie^\pm},\phi)\d\phi-\int_{\partial \B_{n}(1,0)}f( \rho\ie^\pm,\phi)\d\phi\right)\\
			&=-(\area_{n-1}(\partial \B_{n}(1,0)))^{-1}
			\int_{\partial \B_{n}(1,0)}\int_{{2d\ie^\pm}}^{\rho\ie^\pm}\frac{\partial  {f}}{\partial r}(\tau,\phi)\d\tau\d\phi,\label{lm4:1}
		\end{align}
		where $\d\phi$ is a surface element on the unit sphere. Using the Cauchy-Schwarz inequality, we deduce from \eqref{lm4:1}:
		\begin{align*}
			&\big|{\la  f \ra _{\wt \Sigma\ie^\pm} - \la  f\ra _{  \Sigma\ie^\pm}}\big|^2\leq
			(\area_{n-1}(\partial \B_{n}(1,0)))^{-1}
			\int_{\partial \B_{n}(1,0)}\int_{{2d\ie^\pm}}^{\rho\ie^\pm}\left|\frac{\partial  {f}}{\partial r}(\tau,\phi)\right|^2\tau^{n-1}\d\tau\d\phi\cdot
			\int_{{2d\ie^\pm}}^{\rho\ie^\pm} \tau^{1-n}\d\tau\\
			&\leq(\area_{n-1}(\partial \B_{n}(1,0)))^{-1}\|\nabla f\|_{L^2(Y\ie^\pm\setminus\wt Y\ie^\pm)}^2\cdot
			\begin{cases}
				(n-2)^{-1}\big(({2d\ie^\pm})^{2-n}-(\rho\ie^\pm)^{2-n}\big) ,&n\ge 3\\
				\ln \rho\ie^\pm - \ln {2d\ie^\pm},&n=2.
			\end{cases}
		\end{align*}
		Obviously, the above estimate implies \eqref{lemma:aux4:est}.
	\end{proof}
	
	\begin{lemma}\label{lemma:aux5}
		{There is a $C>0$ such that}
		\begin{equation} 	\label{lemma:aux5:est}
			\forall\, f\in H^1(Y\ie^\pm):\quad
			\big|\la  f \ra _{\Sigma\ie^\pm} - \la  f\ra _{B\ie^\pm}\big|^2 \le
			C (\rho\ie^\pm)^{2-n}\|\nabla f\|^2_{L^2(Y\ie^\pm)}.
		\end{equation}	
	\end{lemma}
	\begin{proof}
		The {sought inequality \eqref{lemma:aux5:est}} follows from the estimates
		\begin{gather*}
			\forall\, f\in H^1(Y\ie^\pm):\quad \big|{\la  f\ra _{  \Sigma\ie^\pm}} - \la  f\ra _{Y\ie^\pm}\big|^2 \le
			C (\rho\ie^\pm)^{2-n}\|\nabla f\|^2_{L^2(Y\ie^\pm)},
			\\
			\forall\, f\in H^1(Y\ie^\pm):\quad \big|\la  f \ra _{ B\ie^\pm} - \la  f\ra _{Y\ie^\pm}\big|^2 \le
			C (\rho\ie^\pm)^{2-n}\|\nabla f\|^2_{L^2(Y\ie^\pm)},
		\end{gather*}
		which are proven similarly to   \eqref{lemma:aux3:est}, namely by downscaling the trace estimate.
	\end{proof}
	
	From Lemmata~\ref{lemma:aux1}--\ref{lemma:aux5} and the definition of the numbers $q\ie^\pm$ we immediately get the first main result of this subsection:
	
	\begin{lemma}\label{lemma:main}
		{There is a $C>0$ such that}
		\begin{equation*} 	
			\forall\, f\in H^1(Y\ie^\pm):\quad
			\big|\la  f \ra _{B\ie^\pm} - \la  f\ra _{D\ie^\pm}\big|^2 \le
			C q\ie^\pm \|\nabla f\|^2_{L^2(Y\ie^\pm)}.
		\end{equation*}	
	\end{lemma}
	
	{Before coming to the proof of Theorem~\ref{th1}, we will derive one more auxiliary result which will be needed to prove Theorem~\ref{th2}.}
	\begin{lemma}\label{lemma:main2}
		{There is a $C>0$ such that}
		\begin{gather}\label{lemma:main2:est}
			\forall f\in H^1(M\e): \quad
			\sum_{(i,j)\in\L\e}\|f\|^2_{L^2(T\ije)}\leq
			Ch\e \|f\|^2_{H^1(M\e)}.
		\end{gather}
	\end{lemma}
	\begin{remark}
		The  estimate \eqref{lemma:main2:est} is rough -- in its proof we will only {make} use of the fact that the passages $T\ije$ are short (cf.~\eqref{assump:5}), but we {are not going to} employ the fact that they are also very thin (cf.~\eqref{drho}).
	\end{remark}
	\begin{proof}[Proof of Lemma~\ref{lemma:main2}]
		Recall that $h\e$ is defined in \eqref{assump:5}.
		Let $f\in C^1(\overline{T\ije})$, $(i,j)\in T\ije$. For $x=(x^1,\dots,x^{n-1},t)$ and $z=(x^1,\dots,x^{n-1},h\ije/2) $ with $(x^1,\dots,x^{n-1})\in\D\ie^+ $ and $t\in (-h\ije/2,h\ije/2)$ (consequently, $x\in T\ije^+$ and $z\in S\ie^+$), one has
		$$
		f(x)=f(z)-\int_t^{h\ije/2} {\partial f\over\partial x^n}(x^1,\dots,x^{n-1},\tau)\d \tau,
		$$
		whence
		\begin{align}\notag
			|f(x)|^2&\le 2|f(z)|^2+2(h\ije/2-t)\int_t^{h\ije/2} \left|{\partial f\over\partial x^n}(x^1,\dots,x^{n-1},\tau)\right|^2\d \tau\\
			&\le 2|f(z)|^2+2h\ije\int_{-h\ije/2}^{h\ije/2} \left|{\partial f\over\partial x^n}(x^1,\dots,x^{n-1},\tau)\right|^2\d \tau.\label{lmmain2:1}
		\end{align}
		Integrating \eqref{lmmain2:1} over $t\in (-h\ije/2,h\ije/2)$ and $(x^1,\dots,x^{n-1})\in \D\ie^+$, {and estimating $\|{\partial f\over\partial x^n}\|^2$ by $\|\nabla f\|^2$}, we obtain
		\begin{gather}\label{lmmain2:2}
			\|f\|^2_{L^2(T\ije)}
			\le 2h\ije\|f\|^2_{L^2(S\ie^+)}+2(h\ije)^2\|\nabla f\|^2_{L^2(T\ije)}.
		\end{gather}
		By density argument, the estimate \eqref{lmmain2:2} holds for any $f\in H^1(T\ije)$. One has
		\begin{align}\notag
			\sum_{(i,j)\in\L\e}\|f\|^2_{L^2(T\ije)}&\leq
			Ch\e\bigg(\sum_{i\in\Z\e }\|f\|^2_{L^2(S\ie^+)} +
			\sum_{(i,j)\in\L\e }\|\nabla f\|^2_{L^2(T\ije)}\bigg)
			\\\notag
			&\leq
			Ch\e\bigg(\|f\|^2_{L^2(\Gamma^+)}+
			\sum_{(i,j)\in\L\e }\|\nabla f\|^2_{L^2(T\ije)}\bigg)\\&\leq
			C_1 h\e\bigg( \|f\|^2_{H^1(M^+)}+\sum_{(i,j)\in\L\e }\|\nabla f\|^2_{L^2(T\ije)}\bigg),\label{lmmain2:3}
		\end{align}
		where on the last step we use the continuity of the trace operators $H^1(M^\pm)\to L^2(\Gamma^\pm)$.
		Then the required inequality \eqref{lemma:main2:est} follows immediately from \eqref{lmmain2:3}.
	\end{proof}

	\subsection{Proof of Theorem~\ref{th1}}
	
	Let $f\e\in L^2(M\e)$ and   $u\e\in H^1(M\e)$ be the unique solution to the problem
	$
	\H\e u\e + u\e =f\e.
	$
	One has
	\begin{gather}\label{weak:e}
		\h\e[u\e,u\e]+\|u\e\|^2_{L^2(M\e)}= (f\e,u\e)_{L^2(M\e)}\leq \|f\e\|_{L^2(M\e)}\|u\e\|_{L^2(M\e)},
	\end{gather}
	whence we obtain the standard estimates
	\begin{gather}\label{aprioriestimate}
		\|u\e\|_{L^2(M\e)}\leq \|f\e\|_{L^2(M\e)},\quad
		\h\e[u\e,u\e]\leq \|f\e\|_{L^2(M\e)}^2.
	\end{gather}
	It follows from \eqref{assump:f1}--\eqref{assump:f2} that
	\begin{gather}\label{normsf}
		\begin{array}{r}
			\text{the norms $\|f\e\|_{L^2(M\e)} $ with $\eps\in (0,\eps_0]$ are uniformly bounded}\\[1mm]
			\text{provided that $\eps_0$ is sufficiently small;}
		\end{array}
	\end{gather}
	in the following we assume that $\eps\le \eps_0$.
	Hence \eqref{weak:e} implies the bound
	\begin{gather}\label{ue:bound}
		\h\e[u\e,u\e]+\|u\e\|^2_{L^2(M\e)}\leq C.
	\end{gather}
	{Next we recall} that the operator $\J\e$ is defined by \eqref{Je} {and} denote
	$$
	u\e^\pm\ceq (\J\e u\e)\restr_{\Omega^\pm}\in H^1(\Omega^\pm).
	$$
	It follows {then} from \eqref{ue:bound} that
	\begin{gather}\label{ue:bound+}
		\|u\e^\pm\|_{H^1(\Omega^\pm)}\leq C.
	\end{gather}
	Therefore, there exists a  sequence $\eps_k\searrow 0$ and $u^\pm\in H^1(\Omega^\pm)$ such that
	\begin{gather}\label{ue:weakconv}
		u\e^\pm \rightharpoonup u^\pm\text{ in }H^1(\Omega^\pm)\;\text{ as }\;\eps=\eps_k\to 0.
	\end{gather}
	Our goal is to show that the function $u\in H^1(\Omega\setminus\Gamma)$ given by $u\restr_{\Omega^\pm}=u^\pm$ satisfies
	\begin{gather}\label{limiting:eq}
		\H u + u =f.
	\end{gather}
	Since the equation \eqref{limiting:eq} has a unique solution, we conclude that \eqref{ue:weakconv} holds not only on the subsequence $\eps_k$, but for the whole family $\{u\e\}_{\eps}$.
	
	One  has the   identity
	\begin{gather}\label{integral:identity}
		\h\e[u\e,w\e]+(u\e,w\e)_{L^2(M\e)}=(f\e,w\e)_{L^2(M\e)},
	\end{gather}
	which holds for an arbitrary $w\e\in H^1(M\e)$.
	Our strategy is to insert a specially chosen test function $w\e$ into~\eqref{integral:identity} and to pass to the limit as $\varepsilon \to 0$ 
	with the aim to arrive on the integral identity \eqref{limiting:eq}.
	\smallskip
	
	We introduce the following subset $\mathscr{V}$ of $H^1(\Omega\setminus\Gamma)$:
	\begin{gather}
		\mathscr{V}=\big\{w\in H^1(\Omega\setminus\Gamma):\ w\restr_{\Omega^\pm}\in C^\infty(\overline{\Omega^\pm}) \big\}.\label{V}
	\end{gather}
	The set $\mathscr{V}$ is dense in $H^1(\Omega\setminus\Gamma)$ (with respect to $H^1$-norm).
	
	Let $w$ be an arbitrary real-valued function from $\mathscr{V}$. 
	The assumption that it is real-valued is made solely to simplicity the presentation (now, we do not have to deal with conjugation signs); for an arbitrary $w$ the final result (see the equality \eqref{final:eq}) will follow by linearity.	
	We set $w^\pm\coloneqq w\restr_{\Omega^\pm}$ {and note} that $w^\pm\in C^\infty(\overline{\Gamma})$. Let $\phi:\R\to\R$ be a smooth function satisfying
	\begin{gather*}
		\phi(t)=1\;\text{ for }\;t\le 1/2,
		\quad
		\phi(t)=0\;\text{ for }\;t\ge 1.
	\end{gather*}
	Then we define the function $w\e\in H^1(M\e)$ by
	\begin{gather}\label{we:test}
		w\e(x)=
		\begin{cases}
			\ds\wt w^\pm(x)+\sum_{i\in\I\e}(\wt w^\pm(\wt x\ie^\pm)-\wt w^\pm(x))\phi\ie^\pm(x)
			,&x\in M^\pm, \\
			\ds h\ije^{-1} (\wt w^+(\wt x\ie^+)-\wt w^-(\wt x\je^-) )\left(x^n+\frac{h\ije}2\right)+\wt w^-(\wt x\je^\pm),
			&x=(x^1,\dots,x^n)\in T\ije
		\end{cases}
	\end{gather}
	(recall that $x^n\in (- {h\ije}/2, {h\ije}/2)$ for $x\in T\ije$),
	where $\wt w^\pm\in C^\infty (\overline{M^\pm})$ are defined by
	\begin{gather}\label{wtw}
		\wt w^\pm(x^1,\dots,x^{n-1},x^n)=
		\begin{cases}
			w(x^1,\dots,x^{n-1},x^n-1),&  x=(x^1,\dots,x^n)\in M^+,\\
			w(x^1,\dots,x^{n-1},x^n+1),&  x=(x^1,\dots,x^n)\in M^-,
		\end{cases}
	\end{gather}
	and the cut-off functions $\phi\ie^\pm:\R^n\to\R$ are given by
	\begin{gather*}
		\phi\ie^\pm(x)\ceq \phi\bigg(\frac{|x-\wt x\ie^\pm|}{\rho\ie^\pm}\bigg).
	\end{gather*}
	Due to \eqref{2drho}, the cut-off function $\phi\ie^\pm$ is equal to $1$ on $D\ie^\pm$. Consequently, we get
	$$
	w\e\restr_{D\ie^+}=\wt w^+(\wt x\ie^+)=w\e\restr_{S\ie^+},\quad
	w\e\restr_{D\je^-}=\wt w^-(\wt x\je^-)=w\e\restr_{S\je^-},\quad (i,j)\in\L\e,
	$$
	{which} is why $w\e$ indeed belongs to $H^1(M\e)$.
	
	From the definition of the cut-off functions $\phi\ie^\pm$ we immediately conclude {that}
	\begin{gather*}
		\mathrm{supp}(\phi\ie^\pm)\cap \{x\in \R^n:\ \pm x^n>1\}\subset
		\overline{Y\ie^\pm},
	\end{gather*}
	where $Y\ie^\pm$ are given by \eqref{Yie}; {in view of} \eqref{assump:2}, the sets $Y\ie^\pm$ are pairwise disjoint. Taking these observations into account, we insert $w\e$ into the equation in \eqref{integral:identity} and obtain
	\begin{gather}\label{I:all}
		I_{1,\eps}^++I_{1,\eps}^-+
		I_{2,\eps}^++I_{2,\eps}^-+
		I_{3,\eps}+I_{4,\eps}=
		I_{5,\eps}^++I_{5,\eps}^-+
		I_{6,\eps}^++I_{6,\eps}^-+I_{7,\eps}.
	\end{gather}
	where
	\begin{align*}
		I_{1,\eps}^\pm&\ceq \int_{M^\pm}\left\{\nabla u\e(x)\cdot\nabla \wt w^\pm(x) + u\e(x) \wt w^\pm(x)\right\}\d x ,\\
		I_{2,\eps}^\pm&\ceq \sum_{i\in\I\e}\int_{Y\ie^\pm}\left\{\nabla u\e(x)\cdot\nabla \big( (\wt w^\pm(\wt x\ie^\pm)-\wt w^\pm(x))\phi\ie^\pm(x)\big) + u\e(x)   (\wt w^\pm(\wt x\ie^\pm)-\wt w^\pm(x) )\phi\ie^\pm(x)\right\}\d x ,\\
		I_{3,\eps}&\ceq \sum_{(i,j)\in\L\e}\int_{T\ije}\nabla u\e(x)\cdot \nabla w\e(x) \d x ,\\
		I_{4,\eps}&\ceq \sum_{(i,j)\in\L\e} \int_{T\ije} u\e(x) w\e(x)\d x ,\\
		I_{5,\eps}^\pm&\ceq \int_{M^\pm}f\e(x)  \wt w^\pm(x) \d x,\\
		I_{6,\eps}^\pm&\ceq \sum_{i\in\I\e}\int_{Y\ie^\pm}f\e(x)
		\big( \wt w^\pm(\wt x\ie^\pm)-\wt w^\pm(x)\big)\phi\ie^\pm(x)  \d x ,\\
		I_{7,\eps}&\ceq \sum_{(i,j)\in\L\e} \int_{T\ije} f\e(x) w\e(x)\d x.
	\end{align*}
	\smallskip
	
	{Next} we {will} investigate the above terms {one by one}. Using \eqref{ue:weakconv} and taking into account \eqref{Je} and \eqref{wtw}, we get
	\begin{gather}\label{I1:final}
		I_{1,\eps}^\pm=( u\e^+,w^+)_{H^1(\Omega^+)}+( u\e^-,w^-)_{H^1(\Omega^-)}\underset{\eps\to 0}\to
		(u,w^+)_{H^1(\Omega^+)}+(u,w^-)_{H^1(\Omega^-)} .
	\end{gather}
	Similarly, \eqref{assump:f2} {yields}
	\begin{gather}\label{I5:final}
		I_{5,\eps}^\pm=(\J\e f\e,w^+)_{L^2(\Omega^+)}+(\J\e f\e,w^-)_{L^2(\Omega^-)}\underset{\eps\to 0}\to
		(f,w^+)_{L^2(\Omega^+)}+(f,w^-)_{L^2(\Omega^-)} .
	\end{gather}
	{Moreover}, we observe that
	\begin{gather}\label{ww:est}
		\text{for }\;x\in Y\ie^\pm\;\text{ one has }\; |\wt w^\pm(\wt x\ie^\pm)-\wt w^\pm(x)|\leq C \rho\ie^\pm,
	\end{gather}
	where the constant $C$ depends on the functions $w^\pm$ {only}. Furthermore, we have
	\begin{gather}\label{phi:est}
		|\phi^\pm\ie(x)|\leq C_1,\quad |\nabla\phi\ie^\pm(x)|\leq C_2(\rho\ie^\pm)^{-1}.
	\end{gather}
	Combining \eqref{ww:est}--\eqref{phi:est}, we arrive at the estimates
	\begin{align}
		x\in Y\ie^\pm:\quad |\nabla \big( (\wt w^\pm(\wt x\ie^\pm)-\wt w^\pm(x))\phi\ie^\pm(x)\big)|\leq C,
		\
		| (\wt w^\pm(\wt x\ie^\pm)-\wt w^\pm(x))\phi\ie^\pm(x) |\leq C\rho\ie^\pm,\label{ww:est+}
	\end{align}
	where the constant $C$ depends on $w^\pm$ {only}. Using \eqref{ww:est+} and the Cauchy-Schwarz inequality, we obtain
	\begin{gather}\label{I2:final}
		|I_{2,\eps}^\pm|\leq C\|u\e\|_{H^1(\cup_{i\in\I\e} Y\ie^\pm)}\big(\vol_n(\cup_{i\in\I\e} Y\ie^\pm)\big)^{1/2}\leq
		C_1(\rho\e^\pm)^{1/2}
		\to 0\;\text{ as }\;\eps\to 0.
	\end{gather}
	where {in} the penultimate step {we used} \eqref{ue:bound+} and the fact that the sets  $Y\ie^\pm$ are pairwise disjoint and belong to the $\rho\e^\pm$-neighborhood of  $\Gamma^\pm$ (recall that $\rho\e^\pm$ are defined in \eqref{assump:1}). Similarly,
	using \eqref{normsf}, we have
	\begin{gather}\label{I6:final}
		|I_{6,\eps}^\pm|\leq C\|f\e\|_{L^2(\cup_{i\in\I\e} Y\ie^\pm)}\big(\vol_n(\cup_{i\in\I\e} Y\ie^\pm)\big)^{1/2}\leq
		C_1\big( \vol_n(\cup_{i\in\I\e} Y\ie^\pm)\big)^{1/2}\to 0\;\text{ as }\;\eps\to 0.
	\end{gather}
	It is easy to see that $|w\e(x)|\leq \max\left\{|w^+(x\ie)|,\,|w^-(x\je)|\right\}\leq C$ for $x\in T\ije$; here the constant $C$ depends on $w^\pm$ {only}. Hence, we get
	\begin{align}\notag
		|I_{4,\eps}|&\leq C\|u\e\|_{L^2(\cup_{(i,j)\in\L\e} T\ije)}  \big(\vol_n(\cup_{(i,j)\in\L\e} T\ije)\big)^{1/2}
		\\&
		\leq
		C_1\|u\e\|_{L^2(M\e)} \bigg( h\e\sum_{i\in\I\e} \vol_{n-1}(\D\ie^+)\bigg)^{1/2}\to 0\;\text{ as }\;\eps\to 0.\label{I4:final}
	\end{align}
	where the last {conclusion} follows   from  \eqref{ue:bound}, \eqref{assump:5} and the fact that the sets $\D\ie^+$ are pairwise disjoint and belong to $\Ga$ (whence, their total volume is bounded). Similarly, taking into account \eqref{normsf}, we get 
	\begin{gather}\label{I7:final}
		|I_{7,\eps}|\leq C\|f\e\|_{L^2(\cup_{(i,j)\in\L\e} T\ije)}  \big(\vol_n(\cup_{(i,j)\in\L\e} T\ije)\big)^{1/2} \to 0.
	\end{gather}
	
	Now, we turn our attention to the most intricate term, $I_{3,\eps}$. Taking into account that $\Delta w\e=0$ in $T\ije$, and that the normal derivative of $w\e$ on the lateral part of $\partial T\ije$ vanish, while on the top face (respectively, the bottom face) the derivative with respect to outward normal equals $h\ije^{-1} (\wt w^+(\wt x\ie^+)-\wt w^-(\wt x\je^-) )$ (respectively,  $-h\ije^{-1} (\wt w^+(\wt x\ie^+)-\wt w^-(\wt x\je^-) )$), we obtain via integration by parts
	\begin{align}\notag
		I_{3,\eps}&=\sum_{(i,j)\in\L\e}(\wt w^+(\wt x\ie^+)-\wt w^-(\wt x\je^-) )h\ije^{-1}
		\left(\int_{S\ie^+}u\e\d s_x - \int_{S\je^-}u\e\d s_x\right)
		\\\label{I3:1}
		&=\sum_{(i,j)\in\L\e}K\ije \big(w^+(x\ie^+)-w^-(x\je^-) \big)
		\big(\la u\e\ra_{D\ie^+}-\la u\e\ra_{D\je^-}\big),
	\end{align}
	where {in the second equality} we use \eqref{leps}, the identifications $D\ie^+\sim S\ie^+$, $D\je^-\sim S\je^-$ {with} $(i,j)\in\L\e$, the definitions of the numbers $K\ije$ and the functions $\wt w^\pm$.
	
	For any $\delta>0$, let $u_\delta$ be a function from $\mathscr{V}$ (see \eqref{V})  satisfying
	\begin{gather}\label{delta}
		\|u^\pm-u_\delta^\pm\|_{H^1(\Omega^\pm)} < \delta,\quad\text{where } u_\delta^\pm\ceq u_\delta\restr_{\Omega^\pm}.
	\end{gather}
	By the continuity of the trace operator $H^1(\Omega^\pm)\to L^2(\Gamma)$, we obtain from \eqref{delta}
	\begin{gather}\label{delta+}
		\|u^\pm-u_\delta^\pm\|_{L^2(\Gamma)} < C\delta
	\end{gather}
	{with} the constant $C$ {depending} on $\Omega^\pm$ and $\Gamma$ {only}.
	Also, we introduce the function $v_\delta(x,y)$ on $\Gamma\times\Gamma$  by
	$$
	v_\delta(x,y)=\big( w^+(x)- w^-(y) \big)
	\big( u^+_\delta(x)- u^-_\delta(y)\big),\quad (x,y)\in\Gamma\times\Gamma.
	$$
	Note that
	since $w,u_\delta\in\mathscr{V}$, we have $v_\delta\in C^\infty(\overline{\Gamma}\times\overline{\Gamma})$.
	Using $u_\delta$ and $v_\delta$, we rewrite \eqref{I3:1} as follows,
	\begin{align*}
		I_{3,\eps}&=I_{31,\eps}+I_{32,\eps}+I_{33,\eps}+I_{34,\eps}.
	\end{align*}
	Here
	\begin{align*}
		I_{31,\eps}&\ceq\sum_{(i,j)\in\L\e}K\ije \big( w^+(x\ie^+)-w^-(x\je^-) \big)
		\big(\wt u^+_\delta(\wt x\ie^+)-\wt u^-_\delta(\wt x\je^-) \big)=
		\sum_{(i,j)\in\L\e}K\ije v_\delta(x\ie^+,x\je^-),
		\\
		I_{32,\eps}&\ceq\sum_{(i,j)\in\L\e}K\ije \big( w^+( x\ie^+)-w^-(  x\je^-) \big)
		\big(\la \wt u^+_\delta\ra_{B\ie^+}-\wt u^+_\delta(\wt x\ie^+)-\la \wt u_\delta^-\ra_{B\je^-}+\wt u_\delta^-(\wt x\je^-)\big)
		\\
		I_{33,\eps}&\ceq\sum_{(i,j)\in\L\e}K\ije \big(w^+(x\ie^+)-w^-(x\je^-) \big)
		\big(\la u\e - \wt u_\delta^+\ra_{B\ie^+}-\la u\e\ - \wt u_\delta^-\ra_{B\je^-}\big),
		\\
		I_{34,\eps}&\ceq\sum_{(i,j)\in\L\e}K\ije \big(w^+(x\ie^+)-w^-(x\je^-) \big)
		\big(\la u\e \ra_{D\ie^+}-\la u\e \ra_{B\ie^+}-\la u\e\ra_{D\je^-}+
		\la u\e \ra_{B\je^-}\big),
	\end{align*}
	where $\wt u_\delta^\pm$ are defined by   \eqref{wtw} with $w^\pm$ being replaced by $ u_\delta^\pm$
	(recall  that $x\ie^\pm=(\x\ie,0)$, $\wt x\ie^\pm=(\x\ie,\pm 1)$). 
	By \eqref{assump:main:2}, we obtain
	\begin{gather}\label{I31:final}
		\lim_{\eps\to0}I_{31,\eps}=\int_{\Gamma}K(x,y)v_\delta(x,y)\d s_x \d s_y.
	\end{gather}
	Since $\wt u_\delta^\pm\in C^\infty(\overline{M^\pm})$, one has
	\begin{gather}\label{udelta:est}
		|\la \wt u_\delta\ra_{B\ie^\pm}-\wt u_\delta(\wt x\ie^\pm)|\leq C\rho\ie^\pm\leq C\rho\e^\pm,
	\end{gather}
	where the constant $C$ depends on $u_\delta$ {only}.
	From   \eqref{assump:main:1}  we infer
	\begin{align}\label{sumKij}
		\suml_{(i,j)\in\L\e}K\ije
		\leq C \sum_{i\in\I\e}\vol_{n-1}
		(\B_{n-1}(\rho\ie^+,\x\ie^+))\leq C_1,
	\end{align}
	where in the last step we use the fact that $\B_{n-1}(\rho\ie^+,\x\ie^+)$
	are pairwise disjoint and belong to $\Ga$ (this follows from  \eqref{assump:2}--\eqref{assump:3}).
	Then, using \eqref{udelta:est} and \eqref{sumKij}, we obtain
	\begin{align*}
		|I_{32,\eps}|\leq C \max\{\rho\e^+,\,\rho\e^-\},
	\end{align*}
	where the constant $C$ depends  on $w^\pm$ and $u^\pm_\delta$; applying \eqref{assump:1}, we arrive at
	\begin{align}\label{I32:final}
		\lim_{\eps\to0} I_{32,\eps}=0.
	\end{align}
	Using the Cauchy-Schwarz inequality, \eqref{delta+}, \eqref{sumKij},
	and the bound
	\begin{gather*}
		K\ije\leq C\min\{\area_{n-1}(B\ie^+),\,\area_{n-1}(B\je^-)\},\ (i,j)\in\L\e,
	\end{gather*}
	which follows from \eqref{assump:main:1},
	we get
	\begin{align*}
		|I_{33,\eps}|&\leq C\bigg[\sum_{(i,j)\in\L\e}K\ije\bigg]^{1/2}
		\bigg[\sum_{(i,j)\in\L\e}K\ije\left(|\la u\e - \wt u_\delta^+\ra_{B\ie^+}|^2+|\la u\e\ - \wt u_\delta^-\ra_{B\je^-}|^2\right)\bigg]^{1/2}\\
		&\le C_1
		\bigg[
		\sum_{i\in\I\e}\frac{K_{i,\ell\e(i),\eps}}{\area_{n-1}(B\ie^+)}\|u\e-\wt u^+_\delta\|^2_{L^2(B\ie^+)}+
		\sum_{j\in\I\e}\frac{K_{\ell\e^{-1}(j),j,\eps}}{\area_{n-1}(B\je^-)}\|u\e-\wt u^-_\delta\|^2_{L^2(B\je^-)}\bigg]^{1/2}
		\\
		&\le C_2
		\left[
		\|u\e-\wt u^+_\delta\|^2_{L^2(\Gamma^+)}+
		\|u\e-\wt u^-_\delta\|^2_{L^2(\Gamma^-)}\right]^{1/2}
		= C_2
		\left(\|u\e^+-u_\delta^+\|_{L^2(\Gamma)}^2+\|u\e^--u_\delta^-\|_{L^2(\Gamma)}^2\right)^{1/2}
		\\
		&\le C_2
		\left(\|u\e^+-u^+\|_{L^2(\Gamma)}^2+\|u\e^--u^-\|_{L^2(\Gamma)}^2\right)^{1/2}
		+
		C_2
		\left(\|u^+-u^+_\delta\|_{L^2(\Gamma)}^2+\|u^--u^-_\delta\|_{L^2(\Gamma)}^2\right)^{1/2}
		\\
		&\le C_2
		\left(\|u\e^+-u^+\|_{L^2(\Gamma)}^2+\|u\e^--u^-\|_{L^2(\Gamma)}^2\right)^{1/2}
		+
		C_3 \delta
	\end{align*}
	Since the trace operators $H^1(\Omega^\pm)\to L^2(\Gamma)$ are  compact,
	the weak convergence \eqref{ue:weakconv} implies the strong convergence
	$$
	\|u\e^\pm-u^\pm\|_{L^2(\Gamma)}\to0\text{ as }\eps\to 0,
	$$
	Thus, we arrive at
	\begin{gather}\label{I33:final}
		\overline{\lim_{\eps\to0}}\:|I_{33,\eps}|\leq C_2\delta.
	\end{gather}
	It is important  that the constant $C_2$ {here} is independent on $\delta$.
	
	Finally, we {shall} investigate the term $I_{34,\eps}$. Here Lemma~\ref{lemma:main} comes {into play}: using it, \eqref{assump:main:1} and \eqref{sumKij} we can estimate $I_{34,\eps}$  as follows,
	\begin{align*}
		|I_{34,\eps}|&\leq C\bigg[\sum_{(i,j)\in\L\e}K\ije\bigg]^{1/2}
		\bigg[\sum_{(i,j)\in\L\e}K\ije\left(
		\big
		|\la u\e \ra_{D\ie^+}-\la u\e \ra_{B\ie^+}|^2+
		|\la u\e \ra_{D\je^-}-\la u\e \ra_{B\je^-}|^2\right)\bigg]^{1/2}
		\\
		&\le
		C_1
		\left(
		\sum_{i\in\I\e}(\rho\ie^+)^{n-1} q\ie^+  \|\nabla u\e\|^2_{L^2(Y\ie^+)}+
		\sum_{j\in\I\e}(\rho\je^-)^{n-1} q\je^- \|\nabla u\e\|^2_{L^2(Y\je^-)}
		\right)^{1/2}\\
		&\leq C_2\bigg(\big(\sup_{i\in \I\e}(\rho\ie^+)^{n-1}q\ie^+ + \sup_{j\in \I\e}(\rho\je^-)^{n-1}q\je^-\big)\h\e[u\e,u\e]\bigg)^{1/2}.
	\end{align*}
	whence, due to  \eqref{assump:4} and \eqref{ue:bound}, we obtain
	\begin{gather}\label{I34:final}
		\lim_{\eps\to0}I_{34,\eps}=0
	\end{gather}
	We set $v (x,y)\ceq (w^+(x)-w^-(y))(u^+(x)-u^-(y))$.
	From \eqref{I31:final}, \eqref{I32:final}, \eqref{I33:final}, \eqref{I34:final}, we obtain the  inequality
	\begin{align}
		&\overline{\lim_{\eps\to 0}}\,\Big|I\e^3 - \int_{\Gamma}K(x,y)v(x,y)\d s_x \d s_y \Big|\leq
		{\lim_{\eps\to 0}}
		\left|I\e^{31}-\int_{\Gamma}K(x,y)v_\delta(x,y)\d s_x \d s_y\right|\notag
		\\
		&+
		{\lim_{\eps\to 0}}\,
		|I\e^{32}|+{\lim_{\eps\to 0}}|I\e^{34}|+
		\overline{\lim_{\eps\to 0}}\,
		|I\e^{33}|+\left|\int_\Gamma K(x,y)(v_\delta(x,y)-v (x,y))\d s_x \d s_y\right|\notag
		\\
		&\leq C\left(\delta +\|v_\delta-v\|_{L^2(\Gamma\times\Gamma)}\right),
		\label{est:with:delta}
	\end{align}
	where
	the constants $C $ is independent on $\delta$. It follows from \eqref{delta+} that
	$$
	v_\delta\to v\text{ in }L^2(\Gamma\times\Gamma)\text{ as }\delta\to 0.
	$$
	Hence, since \eqref{est:with:delta} holds for each $\delta>0$, we conclude
	\begin{gather}\label{I3:final}
		\lim_{\eps\to 0}\, I\e^3 =
		\int_{\Gamma}K(x,y)(w^+(x)-w^-(y))(u^+(x)-u^-(y))\d s_x \d s_y.
	\end{gather}
	Combining \eqref{I:all}, \eqref{I1:final}, \eqref{I5:final},   \eqref{I2:final}, \eqref{I6:final},    \eqref{I4:final},  \eqref{I7:final}, \eqref{I3:final},  we arrive at the equality
	\begin{gather}\label{final:eq}
		\h[u,w]+(u,w)_{L^2(\Omega)}=
		(f,w)_{L^2(\Omega)},
	\end{gather}
	which holds for an arbitrary real-valued function $w\in\mathscr{V}$; by linearity it holds  for any  $w\in\mathscr{V}$. Since $\mathscr{V}$ is dense in $H^1(\Omega\setminus\Gamma)$, the {relation \eqref{final:eq}} holds with any $w\in H^1(\Omega\setminus\Gamma)$, which  is equivalent to \eqref{limiting:eq}. {This concludes the proof of} Theorem~\ref{th1}.
	
	\subsection{Proof of Theorem~\ref{th2}}
	
	To {prove this claim} we utilize the abstract scheme from \cite{IOS89}. Let $\{ H \e\}_{\eps>0}$ be a family of (separable) Hilbert spaces, and $\{\mathscr{B}\e\}_{\eps>0}$ be a family of linear bounded operator acting in these spaces. Also, let $ H $ be another separable Hilbert space, and $\mathscr{B} $ be a linear bounded operator in $ H $. We assume that the operators $\mathscr{B}\e$ and $\mathscr{B}$ are positive, compact and self-adjoint. We denote by $\{\mu\ke\}_{k\in\N}$ (respectively, $\{\mu_{k}\}_{k\in\N}$) the sequence of the eigenvalues of the operator $\mathscr{B}\e$ (respectively, $\mathscr{B}$) {arranged} in {the} descending order and counted with multiplicities. Let $f\ke,\ k\in\N$ be the corresponding eigenvectors normalized by the condition $(f\ke,f_{\ell,\eps})_{H\e}=\delta_{k\ell}$.
	
	\begin{theorem}{\!\!\!\rm\cite{IOS89}}\; \label{th:IOS}
		Assume that the following conditions hold:
		\begin{itemize}
			\item[$\rm (C_1)$] There exist linear bounded operators $\mathscr{L}\e: H \to  H \e$ such that
			\begin{gather*}
				\forall f\in H:\quad
				\lim_{\eps\to 0}\|\mathscr{L}\e f\|_{ H \e}=  \|f\|_{ H }.
			\end{gather*}
			\item[$\rm (C_2)$]  The operator norms $\|\mathscr{B}\e\| $ are bounded uniformly in $\eps$.\smallskip
			\item[$\rm (C_3)$]
			$\forall f\in H:\quad \lim\limits_{\eps\to 0}\|\mathscr{B}\e \mathscr{L}\e f-\mathscr{L}\e\mathscr{B} f\|_{ H \e}=0$.
			\item[$\rm (C_4)$] For any family of functions $\{f\e\in H\e\}_{\eps}$ satisfying $\sup_{\eps} \|f\e\|_{ H \e}<\infty$ there exist a subsequence $\eps_k\searrow 0$ and $w\in H$ such that
			\begin{gather*}
				\lim_{\eps_k\to 0}\|\mathscr{B}_{\eps_k} f_{\eps_k}-\mathscr{L}_{\eps_k} w\|_{ H _{\eps_k}}=0.
			\end{gather*}
		\end{itemize}
		Then
		\begin{gather}\label{EV}
			\mu\ke\to \mu_k\text{ as } \eps\to 0,\quad \forall  k\in\N.
		\end{gather}
		Moreover, if
		\begin{gather}\label{muj}
			\mu_{j-1}>\mu_{j}=\mu_{j+1}=...=\mu_{j+m-1}>\mu_{j+m},
		\end{gather}
		then for any eigenvector $u$ corresponding to $\mu_j$ there exists $ u\e\in\mathrm{span}\{u_{k,\eps},\ k=j,\dots,j+m-1\}$ such that
		\begin{gather}\label{EF}
			\|  u\e - \mathscr{L}\e u\|_{ H \e}\to 0,\quad
			\eps\to 0.
		\end{gather}
	\end{theorem}
	
	{With this preliminary, we} return to the proof of Theorem~\ref{th2}. Set
	$$
	H\e\ceq L^2(M\e),\quad H\ceq L^2(\Omega),\quad
	\mathscr{B}\e\ceq (\mathscr{H}\e+\Id)^{-1},\quad
	\mathscr{B}\ceq (\mathscr{H}+\Id)^{-1}.
	$$
	The operators $\mathscr{B}\e$ and $\mathscr{B}$ are {obviously} positive, compact and self-adjoint. We also define the operator $\mathscr{L}\e:H\to H\e$ {by the following prescription},
	\begin{gather*}
		(\mathscr{L}\e f)(x^1,\dots,x^{n-1},x^n)=
		\begin{cases}
			f(x^1,\dots,x^{n-1},x^n-1),&  x=(x^1,\dots,x^n)\in M^+ ,\\
			f(x^1,\dots,x^{n-1},x^n+1),&  x=(x^1,\dots,x^n)\in M^-,\\
			0,&x\in \cup_{(i,j)\in \L\e}T\ije.
		\end{cases}
	\end{gather*}
	Note that $\mathscr{L}\e$ is {the} right inverse to the operator $\J\e$ given by \eqref{Je}. Our aim is to show that the above defined {operators} $\mathscr{B}\e$, $\mathscr{B}$ and $\mathscr{L}\e$ satisfy the assumptions of Theorem~\ref{th:IOS}.
	
	One has $\|\mathscr{L}\e f\|_{H\e}=\|f\|_{H}$ and $\|\mathscr{B}\e\|=(\dist(-1,\sigma(\H\e))^{-1}=1$, whence $\rm (C_1)$ and $\rm (C_2)$ hold true. Next, {let us check} $\rm (C_3)$. For $f\in H$, we set $f\e\ceq \mathscr{L}\e f$. One has
	\begin{gather}\label{C3:1}
		\|\mathscr{B}\e \mathscr{L}\e f - \mathscr{L}\e\mathscr{B}f\|^2_{H\e}=
		\|\J\e\mathscr{B}\e f\e -  \mathscr{B}f\|^2_{H}+\sum_{(i,j)\in\L\e}\|\mathscr{B}\e f\e\|^2_{L^2(T\ije)}.
	\end{gather}
	Evidently, $f\e$ satisfies \eqref{assump:f1}--\eqref{assump:f2}. Then, by Theorem~\ref{th1}, we conclude {that}
	\begin{gather}\label{Be:weak}
		\J\e\mathscr{B}\e f\e \rightharpoonup \mathscr{B}f\text{ in }H^1(\Omega\setminus\Gamma)\text{ as }\eps\to0.
	\end{gather}
	By virtue of {Rellich-Kondrashov theorem, relation} \eqref{Be:weak} implies
	\begin{gather}\label{Be:strong}
		\lim_{\eps\to 0}\|\J\e\mathscr{B}\e f\e - \mathscr{B}f\|_{H}=0.
	\end{gather}
	{Furthermore}, using Lemma~\ref{lemma:main2}, \eqref{assump:5} and \eqref{aprioriestimate}, we obtain
	\begin{gather}\label{C3:2}
		\sum_{(i,j)\in\L\e}\|\mathscr{B}\e f\e\|^2_{L^2(T\ije)}\leq
		C h\e \|\mathscr{B}\e f\e\|^2_{H^1(M\e)}\leq
		2C h\e \| f\e\|^2_{H\e}=
		2C h\e \| f\|^2_{H}\to 0\text{ as }\eps\to 0.
	\end{gather}
	Combining \eqref{C3:1}, \eqref{Be:strong} and \eqref{C3:2}, we conclude that {condition} $\rm (C_3)$ is fulfilled.
	
	{It remains to check the validity of} $\rm (C_4)$. Let $\{f\e\in H\e\}_{\eps}$ be a family of functions satisfying $\sup_{\eps} \|f\e\|_{ H \e}<\infty$. We denote $u\e\ceq\mathscr{B}\e f\e$. For the functions $u\e$ the estimates  \eqref{aprioriestimate} hold, whence   $u\e^\pm\ceq (\J\e u\e)\restr_{\Omega^\pm}\in H^1(\Omega^\pm)$ satisfy \eqref{ue:bound+}. Consequently, there exist a sequence $\eps_k\searrow 0$ and $u^\pm\in H^1(\Omega^\pm)$ satisfying \eqref{ue:weakconv}. By {Rellich-Kondrashov theorem again}, \eqref{ue:weakconv} implies
	\begin{gather}\label{upm:conv1}
		\|u\e^\pm - u^\pm\|_{L^2(\Omega^\pm)}\to 0\;\text{ as }\;\eps=\eps_k\to0.
	\end{gather}
	We introduce $w\in H^1(\Omega\setminus\Gamma)$ by $w\restr_{\Omega^\pm}=u^\pm$.
	Then \eqref{upm:conv1} is equivalent to
	\begin{gather}\label{upm:conv2}
		\|\J\e\B\e f\e  - w\|_{H}\to 0\;\text{ as }\;\eps=\eps_k\to0.
	\end{gather}
	Finally, we get
	\begin{gather}\label{upm:conv3}
		\|\B\e f\e  - \mathscr{L}\e w\|_{H\e}^2= \|\J\e\B\e f\e  - w\|_{H}^2+
		\sum_{(i,j)\in\L\e}\|u\e\|^2_{L^2(T\ije)},
	\end{gather}
	{and} using Lemma~\ref{lemma:main2}, \eqref{assump:5} and \eqref{aprioriestimate}, we have (cf.~\eqref{C3:2})
	\begin{gather}\label{upm:conv4}
		\sum_{(i,j)\in\L\e}\|u\e\|^2_{L^2(T\ije)}\to 0\;\text{ as }\;\eps\to 0.
	\end{gather}
	{Combination of} \eqref{upm:conv2}--\eqref{upm:conv4} {justifies} the assumption $\rm (C_4) $.
	
	{Having} checked the fulfillment of {conditions} (C$_1$)-(C$_4$), {we infer from} Theorem~\ref{th:IOS} {that} the properties \eqref{EV} and {\eqref{EF}} hold true. Taking into account that the eigenvalues of the operators $\H\e$ and $\B\e$ (respectively, $\H$ and $\B$) are linked by
	$$
	\mu\ke=(\lambda\ke+1)^{-1} \quad
	\text{(respectively, }\mu_k=(\lambda_k+1)^{-1}\text{)},
	$$
	we conclude that \eqref{EV} implies \eqref{th2:conv1}. Furthermore, the eigenspaces of $\H\e$ and $\B\e$ corresponding to $\lambda\ke$ and $\mu\ke$ do coincide, {and the analogous} statement holds for the eigenspaces of $\H$ and $\B$. Hence, by \eqref{EF}, if the eigenvalue $\lambda_j$ satisfies
	\eqref{lambdaj} (that is, $\mu_j=(1+\lambda_j)^{-1}$ satisfies \eqref{muj}) and $u$ is the corresponding eigenfunction, then there exists
	$u\e\in\mathrm{span}\{u_{k,\eps},\ k=j,\dots,j+m-1\}$ such that
	\begin{gather} \label{IL1}
		\|u\e-\mathscr{L}\e u\|_{H\e}\to 0\text{ as }\eps\to 0.
	\end{gather}
	Also, we have
	\begin{gather} \label{IL2}
		\|\mathscr{J}\e u\e-u\|^2_{H}=
		\|u\e-\mathscr{L}\e u\|_{H\e}^2-\sum_{(i,j)\in\L\e}\|u\e\|^2_{L^2(T\ije)}
		\leq \|u\e-\mathscr{L}\e u\|_{H\e}^2.
	\end{gather}
	Combining \eqref{IL1} and \eqref{IL2} we arrive at \eqref{th2:conv2}, {by which the} theorem is proven.

	\subsection{Proof of Theorem~\ref{th3}}
	
	We {restrict ourselves to} a sketch of the {argument only}, as {its} ideas closely resemble those used in the proof of Theorem~\ref{th1}. We denote
	$$
	v\e(t)\ceq {\exp(-\H\e t)f\e} =\sum_{k=1}^\infty \exp(-\lambda\ke t)f\ke u\ke,\text{ where }f\ke\ceq (f\e,u\ke)_{L^2(M\e)}.
	$$
	(recall that $\{\lambda\ke,\ k \in \mathbb{N}\}$ is the sequence of eigenvalues of   $\mathscr{H}\e$, and $\{u\ke,\ k \in \mathbb{N}\}$ is the associated sequence of orthonormalized eigenfunctions).
	The assumption $f\e\in H^1(M\e)$ (that is, $\sum_{k=1}^\infty(1+\lambda\ke)|f\ke|^2<\infty$) implies
	$$
	\forall T>0:\quad v\e\in C([0,T];\,H^1(M\e)),\quad \partial_t u\e\in L^2([0,T];\,L^2(M\e)).
	$$
	Furthermore, the following  equalities hold (below $\nabla$ refers to the space derivatives):
	\begin{gather}
		\label{apri1}  \|v\e(t)\|^2_{{L^2(M\e)}}+2\int_0^t\|\nabla
		v\e(\tau)\|^2_{{L^2(M\e)}}\d\tau= \|f\e\|^2_{{L^2(M\e)}},\;\;
		\forall t>0.
		\\
		\label{apri2}  \|\nabla
		v\e(t)\|^2_{{L^2(M\e)}}+2\int_0^t\|\partial_t
		v\e(\tau)\|^2_{{L^2(M\e)}}\d\tau= \|\nabla
		f\e\|^2_{{L^2(M\e)}},\;\; \forall t>0.
	\end{gather}
	We denote $v\e^\pm(t)\ceq(\J\e v\e(t))\restr_{\Omega^\pm}$.
	It follows from \eqref{fe:bound}, \eqref{apri1}, \eqref{apri2} and the definition \eqref{Je} of the operator $\J\e$  that
	\begin{gather}\label{global:bound}
		\|v\e^\pm\|_{H^1(\Omega^\pm\times[0,T])}\le\|v\e\|_{H^1(M\e\times [0,T])}\leq C.
	\end{gather}
	Furthermore, we also have
	$
	|v\e(t)-v\e(s)|^2=\ds\left|\int_s^t \partial_t v\e(\tau)\d\tau\right|^2 \leq
	|t-s|\int_0^{\max\{s,t\}}|\partial_t v\e(\tau)|^2\d\tau,
	$
	whence {in view of} \eqref{fe:bound} and \eqref{apri2}, we {infer that}
	\begin{align*}
		\|v\e(t)-v\e(s)\|^2_{L^2(M\e)}\leq
		|t-s|\int_0^{\max\{s,t\}}\|\partial_t v\e(\tau)\|_{L^2(M\e)}^2\d\tau
		\leq
		{1\over2}|t-s|\|\nabla f\e\|^2_{L^2(M\e)}\leq C|t-s|.
	\end{align*}
	The above inequality together with \eqref{Je} implies that the set $\{v\e^\pm \in C([0,T];L^2(\Omega^\pm))\}_{\eps }$ is equicontinuous.
	{Moreover}, due to \eqref{fe:bound}, \eqref{apri1} {and} \eqref{apri2},  one has $\|v\e(t)\|_{H^1(M\e)}\leq C $, whence by the Rellich-Kondrashov theorem
	{the set $\{v\e^\pm(t)\}_{\eps }$ is relatively compact in $L^2(\Omega^\pm)$ for any $t$. In that case we can use} the {Arzel\`a}--Ascoli theorem (see \cite[Theorem~3.10]{HS91} for its version applicable to $C(E, X)$, where $E$ is a compact topological space and $X$ is a metric space) {to} conclude that
	\begin{gather}\label{Arzela}
		\text{the set of functions }\{v\e^\pm  \}_{\eps}\text{ is relatively compact in }C([0,T],L^2(\Omega^\pm)).
	\end{gather}
	It follows from \eqref{global:bound} and \eqref{Arzela} that there exists a sequence {tending to zero,} $\eps_k\searrow 0$, and $v^\pm\in H^1(\Omega^\pm\times[0,T])$ such that
	\begin{gather}\label{ve:weakconv}
		v\e^\pm \rightharpoonup v^\pm \;\text{ in }\; H^1(\Omega^\pm\times[0,T]),\\
		\label{ve:Cconv}
		\max_{t\in [0,T]}\|v\e^\pm-v^\pm\|_{L^2(\Omega^\pm)}\to 0\quad\text{as }\;
		\eps=\eps_k\to 0.
	\end{gather}
	Let $T>0$. It is easy to see that $v\e$ satisfies the integral identity
	\begin{gather}\label{variat}
		-\int_0^T(v\e(t), w\e)_{L^2(M\e)}\zeta'(t)\d t+ \int_0^T \h\e[v\e(t),w\e]\zeta(t)\d t=0
	\end{gather}
	for {any} $w\e\in H^1(M\e)$ and $\zeta\in C_0^\infty(0,T)$.
	We choose the test function $w\e$ {according to} \eqref{we:test} and pass to the limit in \eqref{variat} as $\eps=\eps_k\to 0$.
	Repeating almost verbatim the arguments we {have used} in the proof of Theorem~\ref{th1} (taking into account the convergence \eqref{ve:weakconv} and the bounds \eqref{apri1}--\eqref{apri2}) we conclude that the function $v\in H^1((\Omega\setminus\Gamma)\times[0,T])$
	given by $v\restr_{\Omega^\pm}=v^\pm$ satisfies
	\begin{gather}\label{variat:hom}
		-\int_0^T(v(t), w)_{L^2(\Omega)}\zeta'(t)\d t+ \int_0^T \h[v(t),w]\zeta(t)\d t=0
	\end{gather}
	for any $\zeta\in C_0^\infty(0,T)$ and $w\in\mathscr{V}$ (the underlying function on top of which $w\e$
	is constructed); recall that $\mathscr{V}$ is given by \eqref{V} and  $\overline{\mathscr{V}}^{H^1}=H^1(\Omega\setminus\Gamma)$. Furthermore, \eqref{assump:f2} and \eqref{ve:Cconv} imply
	\begin{gather}\label{initial:hom}
		v^\pm(0)=f\restr_{\Omega^\pm}.
	\end{gather}
	It is easy to show that the only function $v\in C([0,T];L^2(\Omega))\cap L^2([0,T],H^1(\Omega\setminus\Gamma))$ satisfying  \eqref{variat:hom}--\eqref{initial:hom} is given by $v(t)={\exp(-\H t)f}$, i.e., $v$ is a weak solution to the problem
	\begin{gather*}
		{\ds{\partial v\over \partial t}+\H v=0},\;\; t>0,\qquad
		v(0)=f.
	\end{gather*}
	Since the limiting function $v$  is independent on the subsequence $\eps_k$, we conclude that \eqref{ve:Cconv} holds for the whole family $\{u\e\}_{\eps}$. {By that, the} theorem is proven.

	\section{Non-local Robin boundary conditions}
	\label{sec:5}
	
	{Similarly as local contact interactions have a `one-sided' counterpart in Robin boundary conditions, let us examine in} this section Robin-type conditions {related to the} non-local interface {coupling introduced} in Section~\ref{sec:2}.
	
	Let $\Omega$ be a bounded Lipschitz domain whose boundary contains a relatively open subset $\Gamma$ of a hyperplane $\{x=(x^1,\dots,x^n)\in\R^n:\ x^n=0\}$. We also define the set $\Ga\subset\R^{n-1}$ by \eqref{Ga}. To be specific, we assume that $\Omega$ lies (locally) below $\Gamma$, i.e. there is a $d>0$ such that $\Ga\times (-d,0)\subset\Omega$, see {the left part of} Figure~\ref{fig3}.
	
	In $L^2(\Omega)$ we consider the sesquilinear form $\a$ given by
	\begin{align*}
		\a[u,v]&=\int_{\Omega}\nabla u\cdot\overline{\nabla v}\d x
		+
		\int_\Gamma\int_\Gamma K(x,y)(u (x)-u (y))\overline{(v (x)-v (y))}
		\d s_x \d s_y,\ \dom(\a)=H^1(\Omega ),
	\end{align*}
	where  $K:\overline{\Gamma}\times\overline{\Gamma}\to\R$ is a continuous function satisfying \eqref{K:prop}, $\d s_*$ is  the surface element on $\Gamma$ (the subscript indicates the {integration} variable); hereinafter we use the same notation for $u\in H^1(\Omega )$ and its trace on $\Gamma$. The form $\a$ is symmetric, densely defined, positive, and closed. We denote by $\A$ the self-adjoint and positive operator   associated with $\a$. The function $ u\ceq(\A+\Id)^{-1}f$ with $f\in L^2(\Omega)$ is the weak solution to the problem
	\begin{gather*}
		\begin{cases}
			-\Delta u  + u  = f &\text{in }\Omega ,\\
			\ds\frac{\partial u }{\partial x^n} =-\int_\Gamma K(x,y)(u (x)-u (y)){\d s_y}&\text{on }\Gamma,\\[2mm]
			\ds\frac{\partial u}{\partial \nu}=0&\text{on }\partial\Omega\setminus\Gamma.
		\end{cases}
	\end{gather*}
	Our goal is to approximate the resolvent $(\A+\Id)^{-1}$ of $\A$ by the resolvent  of the Neumann Laplacian $\A\e$ on an appropriately constructed   manifold.
	\begin{figure}[ht]
		\centering
		\raisebox{-0.5\height}{\includegraphics[width=0.4\textwidth]{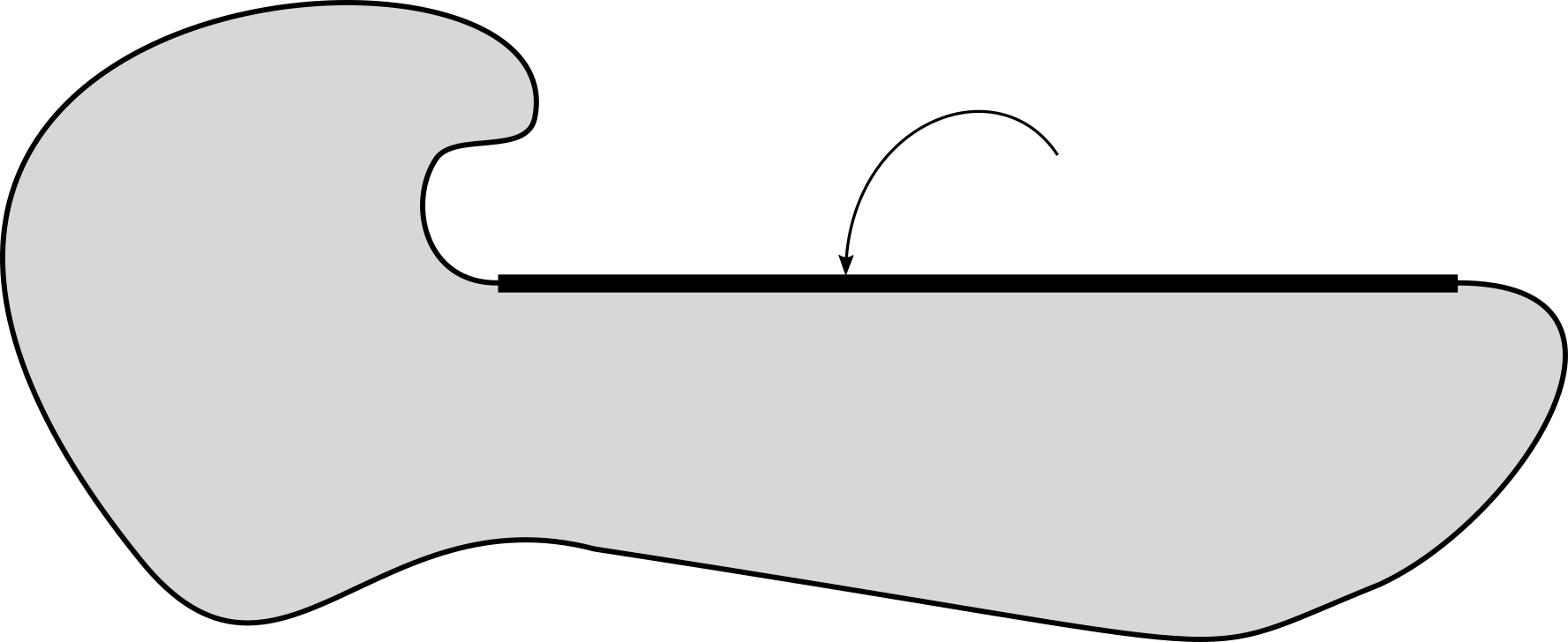}}\qquad\qquad\qquad
		\raisebox{-0.5\height}{\includegraphics[width=0.4\textwidth]{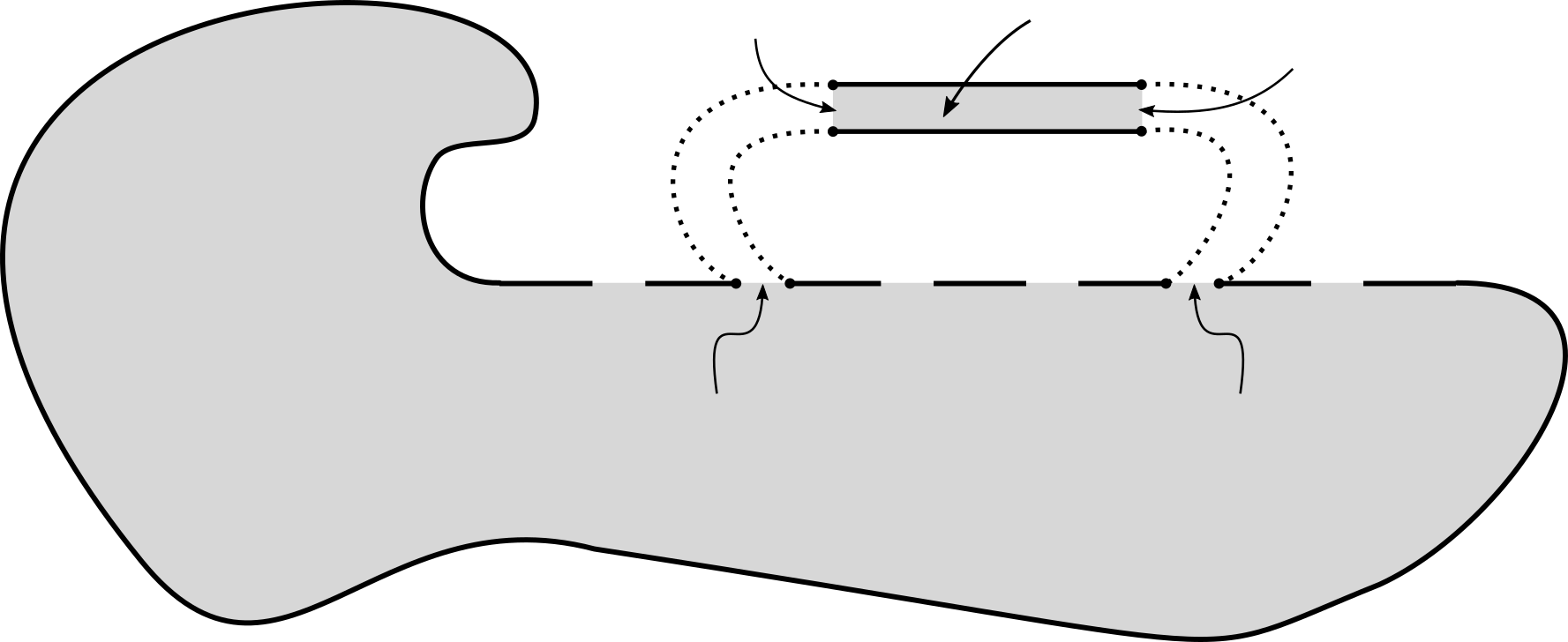}}
		\begin{picture}(0,0)
			
			\put(-303,15){$\Gamma$}
			\put(-65,35){$T\ije$}
			
			\put(-100,33){$D\ie$}
			\put(-38,30){$D\je$}
			\put(-102,-18){$S\ie$}
			\put(-45,-18){$S\je$}
			
			\put(-413,16){$\Omega$}
			\put(-170,16){$\Omega$}
			
		\end{picture}
		\caption{
			The domain $\Omega$ (left) with $\Gamma$  represented by a bold line,   and the manifold $M\e$ (right). The dotted lines demonstrate how the passages $T\ije$ are glued to $\Omega$}
		\label{fig3}
	\end{figure}
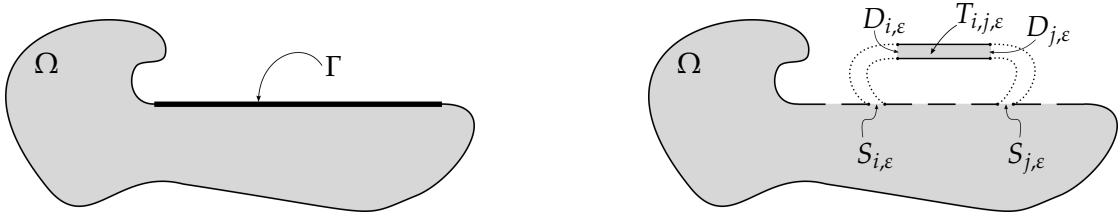
	To construct the {latter}, we first determine the subsets of $\Gamma$ to which we will glue the passages. Let $\eps>0$ be a small parameter, and
	$
	\{\mathbf{D}\ie \subset\Ga,\ i\in \I\e \}
	$
	be a family of pairwise disjoint Lipschitz domains in $\R^{n-1}$ with the set of indices $\I\e$ being finite and having \emph{even number} of elements. We assume that there is a bijective map $\ell\e:\I\e\to \I\e$ having no fixed points (i.e., if $\ell\e(i)=j$, then $i\not= j$), and such that  $\ell\e(i)=j$ iff $\ell\e(j)=i$. In other words, we partition the set $\I\e$ into mutually disjoint pairs, with the function $\ell\e$ defining this partitioning: $i$ and $j$ belong to one pair iff $\ell\e(i)=j$. We assume that the sets $\D\ie$ and $\D\je$ with $i$ and $j$ belonging to one pair are congruent. Finally, we denote $D\ie\ceq\{x\in\R^n:\ (x^1,\dots,x^{n-1})\in \D\ie,\ x^n=0\}\subset\Gamma$, and set $\L\e\ceq \{(i,j)\in\I\e\times\I\e:\  \ell\e(i)=j\}.$
	
	We introduce  the family of passages $\{T\ije\subset\R^n,\ (i,j)\in \L\e\}$ as follows:  we assume that $T\ije=T_{j,i,\eps}\cong\D\ie\times [0,h\ije]\cong\D\je\times [0,h\ije]$  with $0<h\ije=h_{j,i,\eps}<1$, the sets $T\ije$ are pairwise disjoint and have an empty intersection with $\overline\Omega$. We denote by $S\ie$ and $S\je$ the base faces of $T\ije$. Finally, for $(i,j)\in \L\e$, we glue the sets $T\ije$ to $\Omega$ by identifying $S\ie$ and $D\ie$ (respectively, $S\je$ and $D\je$). As a result, we obtain the (topological) manifold
	\begin{gather}\label{Me:Robin}
		M\e=\Omega\cup\Big(\cupl_{(i,j)\in \L\e}
		T\ije\Big)\Big/\sim,
	\end{gather}
	where $\sim$ denotes the above identification. The manifold $M\e$ is depicted {in the right part of} Figure~\ref{fig3}. Its boundary has the form
	$$
	\partial M\e=
	\Big[\partial\Omega \setminus \cupl_{i\in \I\e}D \ie\Big]\cup
	\Big[\cupl_{(i,j)\in \L\e}\left(\partial T\ije\setminus (S\ie \cup S\je ) \right)\Big].
	$$
	We equip it with a metric  assuming that is coincides with the Euclidean metric on each set $\Omega$ and $ { T}\ije$. We introduce the Hilbert spaces $L^2(M\e)$ and $H^1(M\e)$, where the latter consists of functions $u:M\e\to \C$ such that $u\restr_{\Omega}\,\in H^1(\Omega)$ and $u\restr_{T\ije}\,\in H^1(T\ije)$, and the traces from both sides of $S\ie\sim D\ie$ ($i\in\I\e$) do coincide.
	
	Let $d\ie$ and $\x\ie$ be the radius and the center of the smallest ball $\B_{n-1}({d\ie },\x\ie )$ containing  $\D\ie$, {respectively}. We assume that the upscaled sets $\wh\D\ie\ceq d\ie^{-1}D\ie$ satisfy \eqref{assump:shape1}--\eqref{assump:shape2} with the superscript $\pm$ being omitted. Furthermore, we assume that for each $\eps>0$ there exist numbers $ {\rho\ie>0 }$, $i\in \I\e$ satisfying the assumptions \eqref{assump:1}, \eqref{assump:2}, \eqref{assump:4} with the superscript $\pm$ being omitted, and  
	$
	\B_{n}({\rho\ie},x\ie)\cap\{x\in\R^n:\   x^n<0\}\subset\Omega,\; \forall i\in\I\e.
	$ {We also} assume that the condition \eqref{assump:5} holds true. Finally, we define the   quantities $$K\ije\ceq \vol_{n-1}(\D\ie)h\ije^{-1}=\vol_{n-1}(\D\je)h\ije^{-1},\quad (i,j)\in \L\e,$$ and the points $  x\ie\ceq(\x\ie,0)$.  We assume that
	\begin{gather*}
		\forall\eps>0\ \forall (i,j)\in  \L\e:\quad K\ije\leq C\min\big\{(\rho\ie )^{n-1},(\rho\je)^{n-1}\big\},\\\label{assump:main:2Robin}		\sum_{(i,j)\in \L\e} K\ije v( x\ie, x\je)\to\int_{\Gamma\times\Gamma}K(x,y)v(x,y)\d s_x \d s_y\,\text{ as }\,\eps\to0,\; \forall v\in C(\overline{\Gamma}\times\overline{\Gamma}),
	\end{gather*}
	where $K\in C(\overline{\Gamma}\times\overline{\Gamma})$ is the function standing in the definition of $\A$.\smallskip
	
	We are now in position to {state} the result of this section.
	\begin{theorem}\label{th:Robin}
		Let $\{f\e\in L^2(M\e)\}_{\eps}$ be a family of functions satisfying
		\begin{gather*}
			\lim_{\eps\to 0}\sum_{(i,j)\in \L\e}\|f\e\|^2_{L^2(T\ije)}=0,\qquad
			f\e\restr_\Omega\, \rightharpoonup  f\;\text{ in }\;L^2(\Omega)\;\text{ as }\;\eps\to 0,\;\text{ where }\; f\in L^2(\Omega).
		\end{gather*}
		Then, under the above formulated assumptions on the manifold $M\e$, the following holds:
		\begin{gather*}
			(\A\e +\Id )^{-1}f\e\restr_\Omega\, \rightharpoonup (\A+\Id )^{-1}f\;\text{ in }\;H^1(\Omega )\;\text{ as }\;\eps\to0.
		\end{gather*}
		If, in addition,  $ f\e\in H^1(M\e)$ and the bound \eqref{fe:bound} is {satisfied}, then
		\begin{gather}\label{diffusion}
			\forall T>0:\quad \lim_{\eps\to 0} \max_{t\in[0,T]}
			\big\|({\exp(-\A\e t)}f\e)\restr_{\Omega}
			-{\exp(-\A t)}f\big\|_{L_2(\Omega)}=0.
		\end{gather}
		{Finally}, one has
		\begin{gather*}
			\forall k\in\N:\ \lambda\ke\to\lambda_k\;\text{ as }\;\eps\to0,
		\end{gather*}
		where $\{\lambda\ke,\ k\in\N\}$ and $\{\lambda_k ,\ k\in\N\}$ are the sequences of eigenvalues of the operators $\A\e$ and $\A$, respectively, {arranged} in {the} ascending order  counted with multiplicities.
		If the eigenvalue $\lambda_j$ satisfies $\lambda_{j-1}<\lambda_j=\lambda_{j+1}=\dots=\lambda_{j+m-1}<\lambda_{j+m}$  and $u$ is an eigenfunction of $\A$ corresponding to $\lambda_j$, then there exists a sequence $u\e\in\mathrm{span}\{u_{k,\eps},\ k=j,\dots,j+m-1\}$ such that
		\begin{gather*}
			\| u\e\restr_\Omega-u\|_{L^2(\Omega)}\to 0,
		\end{gather*}
		where $\{u\ke,\ k\in\N\}$ is the associated  with $\{\lambda\ke,\ k\in\N\}$  orthonormal sequence of eigenfunctions.
	\end{theorem}
	
	We left the proof {of this claim} to the reader, since it follows the proofs of Theorems~\ref{th1}, \ref{th2}, \ref{th3} almost verbatim, with minor modifications {only}, such as adjustments in the construction of the test function $w\e(x)$.
	
	\smallskip
	
	At the end of this section, we outline an example of a manifold that satisfies the above assumptions. Essentially, this example is as a counterpart to the one constructed in Section~\ref{sec:3}.
	\begin{example}
		We define the sets $\Ga_{s,\eps}$, $\Ga_{s,t,\eps}$, $\D_{s,t,\eps}$ with $s,t\in\Z^{n-1}$ in the same way an Section~\ref{sec:3}. Recall that $\Ga_{s,\eps}$ is the cube in $\R^{n-1}$ with a side length $\eps$ and the center at $\eps s$, $\Ga_{s,t,\eps}$ is a cube with  a side length $L^{-1}\eps^2$ and center at $\x_{s,t,\eps}\ceq  \eps^2 L^{-1}t+\eps s$, where $L$ is the side length of the smallest cube in $\R^{n-1}$ containing $\Ga $ (without loss of generality, we may assume that this cube has its center at zero), and $\D_{s,t,\eps}$ is the ball of the radius $\al_{s,t,\eps} d\e$ centered at $\x_{s,t,\eps}$ with $d\e>0$ satisfying \eqref{assump:5+}--\eqref{assump:4+} and $\al\ije>0$ being defined by \eqref{al:ije}.  Let $\Z\e\subset \Z^{n-1}$ be a set consisting of $s\in\Z^{n-1}$ such that \eqref{eps:neighb} holds with $(-\eps,0)$ instead of $(-\eps,\eps)$.
		
		Now, we introduce the set of indices $\I\e\ceq (\Z\e\times\Z\e)\setminus\{(s,t)\in \Z\e\times\Z\e:\, s=t\}$, and {associate with} $i\in\I\e$ the sets $\D\ie\ceq \D_{s,t,\eps}$, where $i=(s,t)$ with $s,t\in\Z\e$. {Then we} define the bijective map $\ell\e:\I\e\to \I\e$ by $\ell\e:(s,t)\mapsto (t,s),\ s,t\in\Z\e,\, s\not=t$; cf.~Figure~\ref{fig4}. Finally, we {use again} the numbers $h\ije$ {defined} by \eqref{h:ije}.
		
		It is easy to check that the manifold $M\e$ \eqref{Me:Robin} with the sets $\D\ie$, the map $\ell\e$ and numbers $h\ije$ {introduced} above satisfy all the conditions required in Theorem~\ref{th:Robin}.
		\begin{figure}[ht]\label{fig4}
			\centering
			\includegraphics[width=0.5\textwidth]{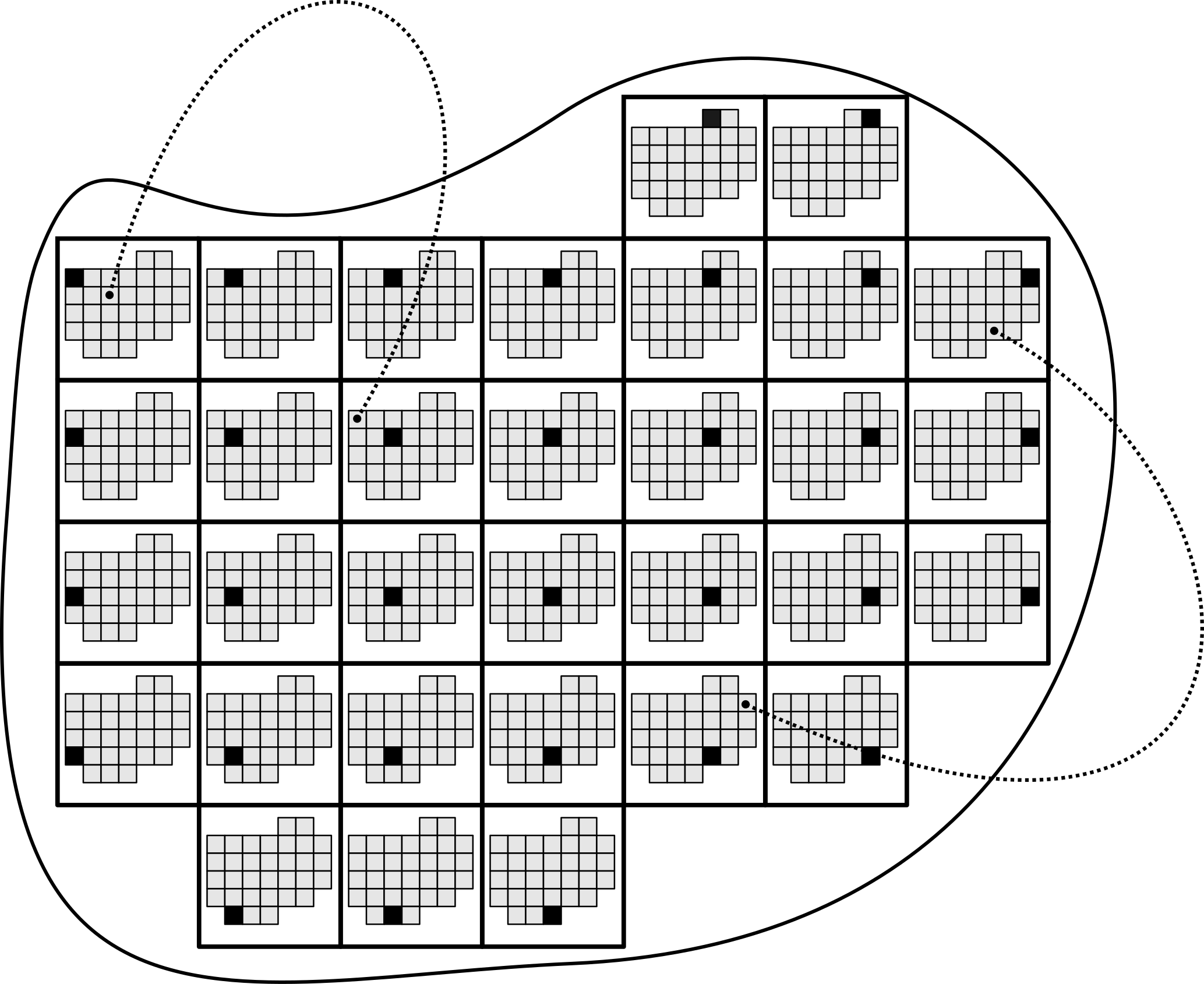}
			\caption{The set $\Ga$.
				The big (respectively, small) cubes correspond to the sets $\Ga_{s,\eps}$
				(respectively, $\Ga_{s,t,\eps}$).
				The dotted lines link the cubes $\Ga_{s,t,\eps}$ and $\Ga_{t,s,\eps}$
				(i.e., those cubes that further will be linked by a passage).
				The black cubes corresponds to $\Ga_{s,t,\eps}$ with $s=t$;
				within these cubes we have no `holes' $\D_{s,t,\eps}$}
		\end{figure}
	\end{example}
	
	\section*{Acknowledgments}
	{The work of P.E. was in part supported by the European Union's Horizon 2020 research and innovation programme under the Marie Sk\l odowska-Curie grant agreement No 873071.} A.K. is grateful to the Excellence Project FoS UHK 2204/2025-2026 for the financial support.
	

\end{document}